\documentclass[11pt]{article}
\usepackage{amsmath}
\usepackage{amsfonts}
\usepackage{amssymb}
\usepackage{mathrsfs}
\usepackage{amsthm}
\usepackage{palatino}
\usepackage[margin=2.5cm, vmargin={1.5cm}]{geometry}

\usepackage{dcolumn}

\usepackage{times}
\usepackage{graphicx}
\usepackage{color}

\usepackage{pstricks}
\usepackage{pst-plot}

\newrgbcolor{LemonChiffon}{1.0 0.98 0.8}
\newrgbcolor{magenta}{1.0 0.3 0.4}
\newrgbcolor{mistyrose}{1.0 0.0 0.6}
\newrgbcolor{deeppink}{1.0 0.08 0.58}
\newrgbcolor{lightsalmon}{1.0 0.63 0.48}
\newrgbcolor{lightred}{0.7 0.0 0.0}
\newrgbcolor{lightblue}{0.68 0.85 0.90}
\newrgbcolor{indianred}{0.80  0.36  0.36}
\newrgbcolor{lightgreen}{0.56 0.93 0.56}
\newrgbcolor{byellow}{0.93 0.76 0.3}

\DeclareMathOperator{\cl}{Cl_2}

\begin{document}

\title{Partitions and Sylvester waves}

\author{Cormac O'Sullivan\footnote{
\newline
This research was supported, in part, by a grant of computer time from the City University of New York High Performance Computing Center under NSF Grants CNS-0855217, CNS-0958379 and ACI-1126113.
Support for this project was also provided by a PSC-CUNY Award, jointly funded by The Professional Staff Congress and The City University of New York.}}

\date{July 15, 2017}

\maketitle

%modular symbol
\def\s#1#2{\langle \,#1 , #2 \,\rangle}

%domains
\def\H{{\mathbf{H}}}
\def\F{{\frak F}}
\def\C{{\mathbb C}}
\def\R{{\mathbb R}}
\def\Z{{\mathbb Z}}
\def\Q{{\mathbb Q}}
\def\N{{\mathbb N}}
%symbols
\def\G{{\Gamma}}
\def\GH{{\G \backslash \H}}
\def\g{{\gamma}}
\def\L{{\Lambda}}
\def\ee{{\varepsilon}}
\def\K{{\mathcal K}}
\def\Re{\mathrm{Re}}
\def\Im{\mathrm{Im}}
\def\PSL{\mathrm{PSL}}
\def\SL{\mathrm{SL}}
\def\Vol{\operatorname{Vol}}
\def\lqs{\leqslant}
\def\gqs{\geqslant}
\def\sgn{\operatorname{sgn}}
\def\res{\operatornamewithlimits{Res}}
\def\li{\operatorname{Li_2}}
\def\lip{\operatorname{Li}'_2}
\def\pl{\operatorname{Li}}

\def\nb{{\mathcal B}}
\def\cc{{\mathcal C}}

\def\clp{\operatorname{Cl}'_2}
\def\clpp{\operatorname{Cl}''_2}
\def\farey{\mathscr F}

\newcommand{\stira}[2]{{\genfrac{[}{]}{0pt}{}{#1}{#2}}}
\newcommand{\stirb}[2]{{\genfrac{\{}{\}}{0pt}{}{#1}{#2}}}
\newcommand{\norm}[1]{\left\lVert #1 \right\rVert}

%theorems

\newtheorem{theorem}{Theorem}[section]
\newtheorem{lemma}[theorem]{Lemma}
\newtheorem{prop}[theorem]{Proposition}
\newtheorem{conj}[theorem]{Conjecture}
\newtheorem{cor}[theorem]{Corollary}
\newtheorem{assume}[theorem]{Assumptions}

\newcounter{coundef}
\newtheorem{adef}[coundef]{Definition}

\renewcommand{\labelenumi}{(\roman{enumi})}
\newcommand{\spr}[2]{\sideset{}{_{#2}^{-1}}{\textstyle \prod}({#1})}
\newcommand{\spn}[2]{\sideset{}{_{#2}}{\textstyle \prod}({#1})}

\numberwithin{equation}{section}

\bibliographystyle{alpha}

\begin{abstract}
The restricted partition function $p_N(n)$ counts the partitions of the integer $n$ into at most $N$ parts. In the nineteenth century Sylvester described these partitions as a sum of waves. We give detailed descriptions of these waves and, for the first time,  show the asymptotics of the initial waves as $N$ and $n$ both go to infinity at about the same rate. This allows us to see when the initial waves are a good approximation to $p_N(n)$ in this situation. Our proofs employ  the saddle-point method of Perron and the dilogarithm.
\end{abstract}

\section{Introduction}
\subsection{Decomposing partitions into  waves}
Let $p(n)$ be the number of  partitions of the integer $n$. This is the number of ways to write $n$ as a sum of non-increasing positive integers.
Also let $p_N(n)$ count the partitions of $n$ with at most $N$ summands. (We are following the notation that was convenient in \cite{Ra,OS15,OS1}.) As usual, $p(0)$ and $p_N(0)$ are defined to be $1$.
Since the work of Cayley \cite{Ca} and Sylvester \cite{Sy3}%
, we know that
\begin{equation} \label{sylthm}
    p_N(n)=\sum_{k=1}^N W_k(N,n) \qquad \quad (N \in \Z_{\gqs 1},\ n \in \Z_{\gqs 0})
\end{equation}
where each $W_k(N,n)$ may be expressed in terms of a  sequence of $k$ polynomials $w_{k,m}(N,x) \in \Q[x]$ for $0\lqs m \lqs k-1$. Similarly to \cite[p. 641]{SZ12} we write
\begin{equation} \label{wkv}
    W_k(N,n) = \bigl[ w_{k,0}(N,n), \ w_{k,1}(N,n), \ \ldots, \ w_{k,k-1}(N,n) \bigr],
\end{equation}
where the notation in \eqref{wkv} indicates that the value of $W_k(N,n)$ is given by one of the polynomials on the right and we select $w_{k,j}(N,n)$ when $n \equiv j \bmod k$. As we will see, the degrees of the polynomials on the right of \eqref{wkv} are at most $\lfloor N/k \rfloor -1$.

For example,  with $N=5$ we have $p_5(n)=W_1(5,n)+ \cdots + W_5(5,n)$
where
\begin{align}
W_1(5,n) & = \left[30 n^4+900 n^3+9300 n^2+38250 n+50651 \right]/86400, \label{w15}\\
    W_2(5,n) & = \left[ 2n+15, \ -2n-15 \right]/128,\notag\\
    W_3(5,n) & = \left[ 2, \ -1, \ -1 \right]/27,\notag\\
    W_4(5,n) & = \left[ 1, \ 1, \ -1, \ -1 \right]/16,\notag\\
    W_5(5,n) & = \left[ 4, \ -1, \ -1, \ -1, \ -1 \right]/25. \notag
\end{align}
This computation was first carried out by Cayley in 1856, \cite[p. 132]{Ca}.
Sylvester called $W_k(N,n)$ the {\em $k$-th wave} and provided the formula
\begin{equation} \label{wave}
    W_k(N,n)=\res_{z=0} \sum_\rho \frac{\rho^n e^{n z}}{(1-\rho^{-1} e^{-z})(1-\rho^{-2} e^{-2z}) \cdots (1-\rho^{-N} e^{-Nz})}
\end{equation}
in \cite{Sy3}. Here $\res_{z=0}$ indicates the coefficient of $1/z$ in the Laurent expansion about $0$, and the sum is over all primitive $k$-th roots of unity $\rho$. For example, Glaisher shows in \cite[Sections 19-30]{Gl} that the first wave is
\begin{equation} \label{w1}
    W_1(N,n) = \frac{(-1)^{N-1}}{ N!} \sum_{j_0+j_1+ \cdots + j_N = N-1}
    \left(-n \right)^{j_0} B_{j_1}   \cdots  B_{j_N} \frac{ 1^{j_1}  \cdots
N^{j_N}}{ j_0 ! j_1 !  \cdots j_N!}
\end{equation}
where the Bernoulli numbers $B_m$ are the constant terms of the Bernoulli polynomials $B_m(t)$, defined with
\begin{equation} \label{bepo}
    \frac{z e^{tz}}{e^z-1}  = \sum_{m=0}^\infty B_m(t) \frac{z^m}{m!}.
\end{equation}
%it is well-known that $B_m(1/2)=(2^{1-m}-1)B_m$.
We will take \eqref{wave} as our definition of $W_k(N,n)$ and  prove in Section \ref{sw} that \eqref{sylthm} and  \eqref{wkv} follow.

Clearly the first wave $W_1(5,n)$ will make the largest contribution to $p_5(n)$ for large $n$.
Similarly,  for any fixed $N$ as $n \to \infty$ we have $p_N(n)  \sim W_1(N,n)$, since the first wave has the biggest degree, and the well-known asymptotic\footnote{As usual, the notation $f(z)=O(g(z))$, or equivalently $f(z) \ll g(z)$, means that there exists a $C$ so that $|f(z)|\lqs C\cdot g(z)$ for all $z$ in a specified range. The number $C$ is called the {\em implied constant}.}
\begin{equation*}
  p_N(n) = \frac{n^{N-1}}{(N-1)!N!} + O(n^{N-2})
\end{equation*}
 then follows from \eqref{w1}.

\subsection{Main results} \label{maru}
 A more difficult question is how $W_1(N,n)$ or the first waves $W_1(N,n)+W_2(N,n)+\cdots$ compare with $p_N(n)$  as $N$ and $n$ both go to $\infty$ together. Suppose, to begin with, that $N=n$. We obtain  the unrestricted partitions $p(n)=p_n(n)$, and see for example that
\begin{align}
    22 = p(8) & = W_1(8,8)+ W_2(8,8)+ W_3(8,8)+ \cdots +  W_{8}(8,8) \notag\\
    & \approx  21.4127 + 0.4112 -0.0566 + \cdots + 0.0625\label{lar}.
\end{align}
So $W_1(n,n)$ gives a good approximation to $p(n)$ for $n=8$ and we may ask if $W_1(n,n) \sim p(n)$ for large $n$.

To examine this question we recall the
 formula of Hardy, Ramanujan and Rademacher, given for example in \cite[Eq. (120.10)]{Ra} as
\begin{equation} \label{hrr}
    p(n)=\sum_{k=1}^\infty \frac{A_k(n)}{2\sqrt{3k}(n-1/24)}\left(\cosh\left(\frac{2\pi}{\sqrt{6}k}\sqrt{n-1/24}\right)-
    \frac{\sinh\left( \frac{2\pi}{\sqrt{6}k}\sqrt{n-1/24}\right)}{\frac{2\pi}{\sqrt{6}k}\sqrt{n-1/24}} \right)
\end{equation}
where $A_k(n)$ may be expressed with Selberg's formula \cite[Eq. (123.2)]{Ra} as
\begin{equation*}
    A_k(n) = \sqrt{k/3}\sum_{\substack{0\lqs \ell \lqs 2k-1 \\ (3\ell^2-\ell)/2 \equiv -n \bmod k}} (-1)^\ell \cos\left( \frac{(6\ell-1)\pi}{6k}\right).
\end{equation*}
The $k=1$ term of \eqref{hrr} is already a good approximation to $p(n)$; using Rademacher's bound in \cite[p. 277]{Ra} for the error after truncation yields, since $A_1(n)=1$,
\begin{equation} \label{pnexp}
    p(n) =\frac1{4\sqrt{3}(n-1/24)} \exp\left(\frac{2\pi}{\sqrt{6}}\sqrt{n-1/24}\right)\left(1+O\left(\frac 1{\sqrt{n}}\right)\right).
\end{equation}

The only  comparison between $W_1(n,n)$ and $p(n)$ we have found in the literature is by Szekeres. He
  compared the expansion of $p(n)$ as a sum of Sylvester waves to the expansion \eqref{hrr},
noting in \cite[p. 108]{Sze51} that the corresponding terms seem to match up for small values of $n$. For example, when $n=8$ the first three terms of \eqref{hrr} give
\begin{equation} \label{lar2}
    p(8) \approx 21.7092+0.3463-0.0896+\cdots.
\end{equation}
 He speculated that the correspondence between \eqref{lar} and \eqref{lar2} would improve as $n$ gets larger.

 In fact we show that  the sum of the first  waves $W_1(n,n)+W_2(n,n)+ \cdots$ does not stay close to $p(n)$ for large $n$. The size of these waves is approximately $\exp(0.068 n)/n^2$ whereas $p(n)$ grows like $\exp(2.565 \sqrt{n})/n$ according to \eqref{pnexp}. More precisely, for the first 100 waves $W_1(n,n)+ \cdots +W_{100}(n,n)$, we prove the following.

\begin{theorem} \label{ma1}
There exist explicit constants $w_0$ and $z_0$ in $\C$ so that as $n \to \infty$
\begin{equation}\label{wknn}
   \sum_{k=1}^{100} W_k(n,n)  =\Re\left[(2  z_0 e^{-3\pi i z_0})\frac{w_0^{-n}}{n^2}\right] +O\left( \frac{|w_0|^{-n}}{n^3}\right).
\end{equation}
\end{theorem}

 The numbers $w_0$ and $z_0$ are coming from the saddle-point method, with $w_0$ the unique solution in $\C$ to $$\li(w)-2\pi i \log(w) = 0,$$
 where $\li(w)$ denotes the dilogarithm function, as described in Section \ref{dise},
and $z_0:= 1+\log(1-w_0)/(2\pi i)$ so that
\begin{equation} \label{w0x}
    w_0=1-e^{2\pi i z_0}, \quad 1/2 < \Re(z_0) < 3/2.
\end{equation}
In \cite{OS3} it is shown that $w_0$ and $z_0$ may be found to any precision and we have
\begin{equation}\label{zeros}
    w_0 \approx 0.916198 - 0.182459 i, \qquad z_0 \approx 1.181475 + 0.255528 i.
\end{equation}
Hence we may equivalently present \eqref{wknn}  as
\begin{equation} \label{wknn2}
    \sum_{k=1}^{100} W_k(n,n) = \frac{ e^{Un}}{n^2}\left(\psi_1 \cdot \sin(\tau_1 + Vn) +O\left( \frac 1{n}\right)\right)
\end{equation}
with
\begin{align} \label{abuv}
       U & :=-\log |w_0| \approx 0.0680762, & V & :=\arg(1/w_0) \approx 0.196576,\\
       \psi_1 & := |2 i  \cdot  z_0 e^{-3\pi i z_0}| \approx 26.8713, &  \tau_1 & := \arg(2 i  \cdot  z_0 e^{-3\pi i z_0}) \approx -3.06816. \label{abuv2}
\end{align}
For large $n$ we see that the first waves trace an oscillating sine wave with amplitude growing exponentially. The period of the oscillation is $2\pi/V \approx 31.96311$ and successive waves increase by a factor of approximately $e^U \approx 1.07045$. The number $w_0$, which is controlling this behaviour, is in fact a zero of the dilogarithm on a non-principal branch and was first identified in \cite{OS15}. The Rademacher coefficients studied in \cite{OS15} have very similar asymptotic properties to Sylvester waves, and indeed they may be expressed in terms of each other - see \cite[Sect. 4]{OS15}.

\begin{table}[ht]
\begin{center}
\begin{tabular}{c | c | c | c | c}
$n$ & $W_1(n,n)$ & $W_2(n,n)$ &  $W_3(n,n)$ &  $W_4(n,n)$\\ \hline
$1000$ & $ 2.41 \times 10^{31}$ & $ 4.09 \times 10^{13}$  & $-3.03\times 10^{7\phantom{0}}$  & $8.14\times 10^{4}$ \\
$1500$ & $ 2.32  \times 10^{39}$ & $ 2.40 \times 10^{17}$  & $\phantom{-}2.49\times 10^{10}$  & $5.52\times 10^{6}$ \\
$2000$ & $ 4.37  \times 10^{53}$ & $ 4.98 \times 10^{23}$  & $-8.22\times 10^{13}$  & $6.98\times 10^{8}$
\end{tabular}
\caption{Relative sizes of the first waves} \label{www}
\end{center}
\end{table}

Theorem \ref{ma1} is  not the best possible result; we expect it to be true with the sum of the first $100$ waves replaced by just the first wave. Numerically, the other waves are much smaller than the first, as shown in Table \ref{www} for example.

Comparing \eqref{wknn2} and \eqref{pnexp}, we see that the  first  waves must eventually become much larger than $p(n)$. In fact $ \psi_1 e^{Un}/n^2$ equals the right side of \eqref{pnexp} for $n \approx 1480$. We see in Table \ref{tbl} that $p(n)$ and $W_1(n,n)$ get further and further apart after this value. On the other hand, for $n \ll 1480$ the first waves are a good approximation to $p(n)$ as we see in Section \ref{good}.

\begin{table}[ht]
\begin{center}
\begin{tabular}{c | c | c | c}
$n$ & $p(n)$ & $W_1(n,n)$ &  $\Re\left[ (2  z_0 e^{-3\pi i z_0})\frac{w_0^{-n}}{n^2} \right]$ \\ \hline
%$1100$ & $1.15\times 10^{33}$ & $\phantom{-}1.15\times 10^{33}$  & $-3.29\times 10^{27}$  \\
$1200$ & $\mathbf{4.62\times 10^{34}}$ & $\mathbf{\phantom{-}4.62\times 10^{34}}$  & $\phantom{-}1.90\times 10^{30}$  \\
$1300$ & $\mathbf{1.61\times 10^{36}}$ & $\mathbf{\phantom{-}1.61\times 10^{36}}$  & $\phantom{-}3.95\times 10^{33}$  \\
$1400$ & $\mathbf{4.90\times 10^{37}}$ & $\mathbf{\phantom{-}5.21\times 10^{37}}$  & $\phantom{-}3.12\times 10^{36}$  \\
$1500$ & $1.33\times 10^{39}$ & $\phantom{-}2.32\times10^{39}$  & $\phantom{-}9.67\times 10^{38}$  \\
$1600$ & $3.24\times 10^{40}$ & $\mathbf{-8.08\times10^{41}}$ &   $\mathbf{-8.93\times 10^{41}}$  \\
$1700$ & $7.18\times 10^{41}$ & $\mathbf{-1.57\times10^{45}}$  & $\mathbf{-1.60\times 10^{45}}$  \\
$1800$ & $1.46\times 10^{43}$ & $\mathbf{-1.20\times10^{48}}$  & $\mathbf{-1.21\times 10^{48}}$
\end{tabular}
\caption{Comparing  $p(n)$, $W_1(n,n)$ and the asymptotics from Theorem \ref{ma1}} \label{tbl}
\end{center}
\end{table}

Our main result is the following, of which Theorem \ref{ma1} is a special case. Instead of restricting to $n=N$ we allow $n$ in a range between $-\lambda^+N$ and $\lambda^+N$.

\begin{theorem} \label{ma2}
Let $\lambda^+$ be a positive real number. Suppose $N \in \Z_{\gqs 1}$ and  $\lambda N \in \Z$ for $\lambda$ satisfying $|\lambda| \lqs \lambda^+$.
Then there are explicit coefficients $a_{0}(\lambda),$ $a_{1}(\lambda), \dots $ so that
\begin{equation} \label{pres}
   \sum_{k=1}^{100} W_k(N,\lambda N) = \Re\left[\frac{w_0^{-N}}{N^{2}} \left( a_{0}(\lambda)+\frac{a_{1}(\lambda)}{N}+ \dots +\frac{a_{m-1}(\lambda)}{N^{m-1}}\right)\right] + O\left(\frac{|w_0|^{-N}}{N^{m+2}}\right)
\end{equation}
as $N \to \infty$ where  $a_{0}(\lambda)=2  z_0 e^{-\pi i z_0(1+2\lambda)}$ and the implied constant depends only on  $\lambda^+$ and $m \gqs 0$.
\end{theorem}

%As we have seen in \eqref{abuv}, $|w_0|^{-N} \approx \exp(0.068 N)$.
The formula for the next coefficient, $a_{1}(\lambda)$, is given in Proposition \ref{propb1s}. These first waves show the same basic oscillating behavior as the $\lambda=1$ case, with period  $2\pi/V \approx 31.96311$ and  increasing by a factor of approximately $e^U \approx 1.07045$ with each $N$:
\begin{equation} \label{presxx}
   \sum_{k=1}^{100} W_k(N,\lambda N) = \frac{e^{UN}}{N^2}\psi_\lambda \cdot \sin\Bigl(\tau_\lambda +VN\Bigr)+ O\left(\frac{e^{UN}}{N^{3}}\right)
\end{equation}
where, consistent with \eqref{abuv2},
\begin{equation} \label{frbr}
    \psi_\lambda:=2|z_0|e^{\pi \Im(z_0)(1+2\lambda)}, \qquad \tau_\lambda:=\arg(i z_0)-\pi \Re(z_0)(1+2\lambda).
\end{equation}
As before, we expect Theorem \ref{ma2} to be true with the left side of \eqref{pres} replaced by just the first wave $W_1(N,\lambda N)$. Instances of  Theorem \ref{ma2}, comparing the right side of \eqref{pres} to the first wave, are displayed in Table \ref{tgz}.

\begin{table}[ht]
{\footnotesize
\begin{center}
\begin{tabular}{cc|cccc|c}
$N$ & $\lambda$ & $m=1$ & $m=2$ & $m=3$ &  $m=5$ & $W_1(N, \lambda N)$  \\ \hline
$3300$ & $1/3$ & $-8.35526\times 10^{90}$ &  $-8.22612 \times 10^{90}$ & $-8.22662 \times 10^{90}$ &   $-8.22663 \times 10^{90}$ &
$-8.22663 \times 10^{90}$\\
$3300$ & $1$ & $-9.05354\times 10^{91}$ &  $-8.97235 \times 10^{91}$ & $-8.97192 \times 10^{91}$ &   $-8.97194 \times 10^{91}$ &
$-8.97194 \times 10^{91}$ \\
$3300$ & $2$ & $-2.02861\times 10^{92}$ &  $-1.88676 \times 10^{92}$ & $-1.89108 \times 10^{92}$ &   $-1.89104 \times 10^{92}$ &
$-1.89104 \times 10^{92}$
\end{tabular}
\caption{The approximations of Theorem \ref{ma2}  to $W_1(N, \lambda N)$.} \label{tgz}
\end{center}}
\end{table}

Taking $m=1$ and $\lambda =r/N$ in Theorem \ref{ma2} easily gives the following corollary.

\begin{cor} \label{cma2}
Fix $r \in \Z$.
Then
\begin{equation} \label{tg45}
   \sum_{k=1}^{100} W_k(N,r) = \Re\left[(2  z_0 e^{-\pi i z_0}) \frac{w_0^{-N}}{N^{2}} \right] + O\left(\frac{|w_0|^{-N}}{N^{3}}\right)
\end{equation}
as $N \to \infty$  for an implied constant depending only on $r$.
\end{cor}

Thus we see that even in the wave expansion of $P_N(1)=1$,  corresponding to $r=1$ in \eqref{tg45}, the first waves are becoming exponentially large with $N$.
So in general we expect some of the individual waves  on the right of \eqref{sylthm} to be much larger than $p_N(n)$, indicating  a lot of cancelation on the right side.

A situation where the first waves are the same size as $p_N(n)$ is given next.

\begin{theorem} \label{ma3}
Let $\lambda$ be  in the range $0.2 \lqs \lambda \lqs 2.9$. Then for  positive integers $N$ and $\lambda N^2$ we have
\begin{equation*}
   p_N\bigl(\lambda N^2\bigr) = \left(\sum_{k=1}^{100} W_k\bigl(N,\lambda N^2\bigr)\right)\Bigl(1  + O\left(e^{-0.1 N}\right)\Bigr)
\end{equation*}
as $N \to \infty$ where the implied constant is absolute.
\end{theorem}

Theorem \ref{ma3} is easily shown using the methods from the proof of Theorem \ref{ma2}. However, a stronger result that replaces the $100$ waves by the first wave and increases the range of $\lambda$ should be possible. Indeed the work in \cite{Sze51} suggests that the correct range for $\lambda$ should be much larger.

\section{Sylvester waves} \label{sw}
\subsection{Basic properties} \label{swq}
As in \cite[Sect. 2.1]{OS1} we define
\begin{equation}\label{qzns}
Q(z;N,\sigma):=\frac{  e^{2\pi i \sigma z}}{(1-e^{2\pi i 1 z})(1-e^{2\pi i2 z}) \cdots (1-e^{2\pi i N z})}
\end{equation}
and it is easy to show that for any $\sigma \in \C$
\begin{align}
Q(z+1;N,\sigma) & = e^{2\pi i \sigma} Q(z;N,\sigma), \label{q1}\\
Q(-z;N,\sigma) & = \overline{Q(\overline{z};N,\overline{\sigma})}, \label{q1b}\\
Q(-z;N,\sigma) & = (-1)^N Q(z;N,N(N+1)/2-\sigma). \label{q2}
\end{align}
Our  goal is to express the Sylvester waves as residues of $Q(z;N,\sigma)$.
As a function of $z$, $Q(z;N,\sigma)$ is meromorphic with all poles contained in $\Q$. More precisely, the set of poles of $Q(z;N,\sigma)$ in $[0,1)$ equals $\farey_N$, the Farey fractions of order $N$ in $[0,1)$.
We write
\begin{equation} \label{defqn}
    Q_{hk\sigma}(N):=2\pi i \res_{z=h/k} Q(z;N,\sigma)
\end{equation}
for $h/k \in \farey_N$. When dealing with residues the next simple relations for $c \in \C$ are useful:
\begin{equation} \label{sres}
    \res_{z=0}f(z)=c \res_{z=0}f(cz), \quad \res_{z=0}f(z+c)= \res_{z=c}f(z), \quad \overline{\res_{z=c}f(z)}= \res_{z=\overline{c}}\overline{f(\overline{z})}.
\end{equation}
With \eqref{q1}, \eqref{q1b} and \eqref{sres} we obtain
\begin{align}
Q_{01\sigma}(N) & = \overline{Q_{01\overline{\sigma}}(N)}, \label{01s}\\
    Q_{(k-h)k\sigma}(N) & =e^{2\pi i \sigma}\overline{Q_{hk\overline{\sigma}}(N)} \qquad (k\gqs 2). \label{defqnx}
\end{align}
A simple exercise with \eqref{sres} also shows, for $h/k \in \farey_N$, $\rho=e^{2\pi i h/k}$ and $\sigma \in \C$, that
\begin{equation} \label{cty}
    \res_{z=0} \frac{e^{2\pi i \sigma h/k} e^{\sigma z}}{(1-\rho^{-1} e^{-z})(1-\rho^{-2} e^{-2z}) \cdots (1-\rho^{-N} e^{-Nz})} = -\overline{Q_{hk(-\sigma)}(N)}.
\end{equation}
From  \eqref{01s}, \eqref{defqnx} and \eqref{cty} we find:
\begin{prop} \label{rain}
For all $k,$ $N \in \Z_{\gqs 1}$ and all $n\in\Z$,
\begin{equation}\label{wct}
    W_k(N,n) = - \sum_{0\lqs h<k, \ (h,k)=1} Q_{hk(-n)}(N).
\end{equation}
\end{prop}

As shown in \cite[Thm. 2.1]{OS1}, we may relate the right side of \eqref{wct} above to the restricted partitions. Briefly,
the generating function
\begin{equation*}
  \sum_{n=0}^\infty p_N(n) q^n =\prod_{j=1}^N \frac{1}{1-q^j}
\end{equation*}
with $q$ replaced by $e^{2\pi i z}$ becomes $Q(z;N,0)$. We therefore have
\begin{equation}\label{wpn}
    \int_w^{w+1} Q(z;N,-n) \, dz = \begin{cases}
p_N(n) & \text{ \ if \ } \quad n \in \Z_{\gqs 0},\\
0 & \text{ \ if \ } \quad n \in \Z_{<0}
\end{cases}
\end{equation}
for any $w\in \C$ with $\Im (w)$ large enough. Integrating $Q(z;N,-n)$ around the rectangle with corners $w+1$, $w$, $\overline{w}$ and $\overline{w+1}$, using \eqref{q1}, \eqref{q2} and Cauchy's residue theorem proves:
\begin{theorem} \label{qn} For $N \in \Z_{\gqs 1}$ and $n \in \Z$ we have
\begin{equation} \label{th1}
 -\sum_{h/k \in \farey_N} Q_{hk(-n)}(N) =  \begin{cases}
p_N(n)
& \text{ \ if \ } \quad 0 \lqs n  \\
0
& \text{ \ if \ } \quad -N(N+1)/2<n<0 \\
(-1)^{N+1} p_N \bigl(-n-N(N+1)/2 \bigr)
& \text{ \ if \ } \quad n \lqs -N(N+1)/2.
\end{cases}
\end{equation}
\end{theorem}
With Proposition \ref{rain} and Theorem \ref{qn} we have shown that $p_N(n)=\sum_{k=1}^N W_k(N,n)$, (i.e. equality \eqref{sylthm}),  follows from the wave definition \eqref{wave}. This result is known as Sylvester's Theorem - see for example \cite[Sect.4]{OS15} and its contained references.
The next result,  showing that \eqref{wkv} is also a consequence of \eqref{wave}, follows for example from \cite[Sections 4-7]{Gl} which uses partial fractions. We give a new proof.

\begin{prop} \label{well}
For each wave $W_k(N,n)$, equation \eqref{wkv} is valid for polynomials $w_{k,m}(N,x) \in \Q[x]$ that have degree at most $\lfloor N/k \rfloor-1$. They satisfy
\begin{equation} \label{pdf}
    w_{k,0}(N,x)+ w_{k,1}(N,x)+\cdots + w_{k,k-1}(N,x)=0 \qquad \quad (k \gqs 2)
\end{equation}
and more generally, with any factorization $k=bc$ and $0\lqs \ell \lqs b-1$,
\begin{equation} \label{pdf2}
    w_{k,\ell}(N,x)+ w_{k,\ell+b}(N,x)+w_{k,\ell+2b}(N,x)+\cdots + w_{k,\ell+(c-1)b}(N,x)=0 \qquad \quad (c \gqs 2).
\end{equation}
\end{prop}
\begin{proof}
Following Apostol in \cite[Eq. (3.1)]{Ap},  write
\begin{equation}\label{apb}
\frac{z}{\rho e^z - 1} = \sum_{m=0}^\infty \beta_m(\rho) \frac{z^m}{m!} \qquad (\rho \in \C).
\end{equation}
Then  \eqref{wave} may be expressed  as
\begin{multline} \label{psc}
    W_k(N,n)  = \frac{(-1)^{N-1}}{N!} \sum_\rho \rho^{-n} \left[ \text{coeff. of }z^{N-1} \right]e^{-n z}
\left(\frac{z}{\rho e^z - 1}\right)\left(\frac{2z
}{\rho^2 e^{2z} - 1}\right) \cdots \left(\frac{N z}{\rho^N e^{Nz} - 1}\right)\\
 = \frac{(-1)^{N-1}}{N!} \sum_{j=0}^{N-1} \frac{(-n)^j}{j!}
\sum_{j_1+j_2+ \cdots + j_N = N-1-j}
\frac{ 1^{j_1} 2^{j_2} \cdots
N^{j_N}}{ j_1 ! j_2 ! \cdots j_N!}
\left(  \sum_\rho \rho^{-n} \beta_{j_1}(\rho)\beta_{j_2}(\rho^{2})  \cdots
\beta_{j_N}(\rho^{N}) \right)
\end{multline}
where $\rho$ is summed over all primitive $k$th roots of unity. It is clear from the form of \eqref{psc} that $W_k(N,n)$ is a polynomial in $n$ where, for fixed $k$ and $N$, the polynomial depends only on $n \bmod k$. Hence
 \begin{multline}  \label{psc2}
    w_{k,m}(N,x) = \frac{(-1)^{N-1}}{N!} \sum_{j=0}^{N-1} \frac{(-x)^j}{j!}\\
    \times
\sum_{j_1+j_2+ \cdots + j_N = N-1-j}
\frac{ 1^{j_1} 2^{j_2} \cdots
N^{j_N}}{ j_1 ! j_2 ! \cdots j_N!}
\left(  \sum_\rho \rho^{-m} \beta_{j_1}(\rho)\beta_{j_2}(\rho^{2})  \cdots
\beta_{j_N}(\rho^{N}) \right).
 \end{multline}
 The inner sum in \eqref{psc2} will be zero if $j_1+ \cdots +j_N$ is too small since $\beta_0(w)=0$ for $w\neq 1$. For $w=1$ we have $\beta_0(1)=1$. Therefore the smallest value of $j_1+ \cdots +j_N$ that gives a possible nonzero value for the inner sum is $N-\lfloor N/k \rfloor$. Hence we may assume $j \lqs \lfloor N/k \rfloor -1$ in \eqref{psc2}, proving the bound for the degree.

We have $\beta_m(1)=B_m$ and, by a formula of Glaisher \cite[Sect. 97]{Gl} (see also \cite[Prop. 3.2]{OS15}),
\begin{equation}
    \beta_m(\xi)  = (-1)^{m-1} m \sum_{j=1}^{m}  \stirb{m}{j} \frac{(j-1)!}{(\xi - 1)^j}
 \qquad (\xi \neq 1) \label{psi1}
\end{equation}
where $\stirb{m}{j}$ is the Stirling number that denotes the number of ways to partition a set of size $m$ into $j$  non-empty subsets.
For fixed $j_1, \dots, j_N$ the value of the inner sum in \eqref{psc2} is therefore  in the field $\Q(\rho)$.
This value  is unchanged under any automorphism of  $\Q(\rho)$, since the primitive $k$th roots will just be permuted. It follows that this value is in the fixed field of all automorphisms and so rational. Finally, \eqref{pdf} and \eqref{pdf2} follow from \eqref{psc2}  and the identity
\begin{equation*}
    \rho^\ell(1+\rho^b+\rho^{2b}+ \cdots + \rho^{(c-1)b})=0 \qquad \quad (c \gqs 2). \qedhere
\end{equation*}
\end{proof}

\subsection{Explicit waves}
The formula \eqref{psc2} gives a convenient expression for the $k$th wave, especially when combined with the result from \cite[Eq. (3.6)]{OS15}:
 \begin{equation}\label{ret}
    \beta_m(\xi)  = k^{m-1} \sum_{j=0}^{k-1} \xi^{j} B_m(j/k)
\end{equation}
for all $m \in \Z_{\gqs 0}$ and all $\xi \in \C$ with $\xi^k=1$ for $k \in \Z_{\gqs 1}$.
Sylvester in \cite{Sy3} and Glaisher in \cite{Gl} developed different descriptions for $W_k(N,n)$. For some details of this and further treatments see Dowker's papers \cite{Dow2,Dow1} and also \cite{Be1,FR02}. In  \cite[Eq. (46)]{FelRub06} the waves are written in terms of  Bernoulli and Eulerian polynomials of higher order. An interesting expression for the first wave $W_1$ that does not involve Bernoulli polynomials has recently appeared in \cite{DV}.

For fixed $k$ and $0 \lqs w \lqs k-1$ set
\begin{equation*}
    s_{m,w}(N):= \sum_{\substack{1 \lqs j \lqs N , \ j \equiv w \bmod k}} j^m.
\end{equation*}
A variation of a result of Glaisher  using the Apostol coefficients \eqref{apb} is the following, given in \cite[Eq. (4.8)]{OS15}.

\begin{theorem} For $k, N \in \Z_{\gqs 1}$, $s :=\lfloor N/k \rfloor$ and $n \in \Z$
\begin{multline}
    W_k(N,n) = \sum_\rho \frac {(-1)^{s-1} \rho^{-n} }{k^{2s} \cdot s! \prod_{1 \lqs r \lqs N-ks} (1-\rho^r)  } \\
    \times
     \sum_{1j_1+2j_2+ \cdots + N j_{N} = s-1}
     \frac{1}{j_1! j_2! \cdots j_N!}
     \left(-n-\frac{N(N+1)}2 - \sum_{w=0}^{k-1}  \frac{ \beta_1(\rho^w) \cdot s_{1,w}(N)}{1 \cdot 1!}\right)^{j_1} \\
    \times
     \left(- \sum_{w=0}^{k-1}  \frac{ \beta_2(\rho^w) \cdot s_{2,w}(N)}{2 \cdot 2!} \right)^{j_{2}} \cdots
     \left(- \sum_{w=0}^{k-1}  \frac{ \beta_N(\rho^w) \cdot s_{N,w}(N)}{N \cdot N!} \right)^{j_{N}} \label{wavek}
\end{multline}
where the outer sum is over all primitive $k$-th roots of unity $\rho$.
\end{theorem}

This means of calculating $W_k(N,n)$ is computationally faster than \eqref{psc2} when $N$ is large and the wave computations in Tables \ref{www}, \ref{tbl}, \ref{tgz}, \ref{tbl22} and  \ref{w2tb}  were carried out using \eqref{wavek}. An efficient means of computing waves that avoids roots of unity is given in \cite{SZ12}.

 For a simple example,
\begin{equation*}
    W_1(6,n) = \Bigl(    12 n^5 + 630 n^4 +12320 n^3 + 110250 n^2 + 439810 n + 598731\Bigr)/1036800.
\end{equation*}
We notice that all the coefficients of $W_1(6,n)$ are positive (the same was true for $W_1(5,n)$ in \eqref{w15}) and this positivity continues in $W_1(N,n)$ for increasing $N$. However, it must eventually fail; as we see from Table \ref{tbl}, at least one of the coefficients of the polynomial $W_1(1600,n)$ is negative.

For $k=2$ we have $W_2(N,n)=(-1)^n w_{2,0}(N,n)$ by \eqref{pdf}. In this case $\beta_j(-1)=(2^j-1)B_j$ and we find, for instance with $N=10$,
\begin{equation*}
    W_2(10,n) =
      (-1)^n\Bigl(   30 n^4+ 3300 n^3 +125400 n^2 + 1905750 n +9406331 \Bigr)/88473600.
\end{equation*}
The third wave for $N=10$ is
\begin{equation*}
    W_3(10,n) =
      \Bigl[6 n^2 + 344 n + 4317, \
    -28 n - 770, \ -6 n^2 - 316 n - 3547\Bigr]/52488.
\end{equation*}

For large values of $k$ we may simplify the formula \eqref{wavek}. When $N/2<k\lqs N$ we have $s=1$ and hence
\begin{equation}\label{ptk}
    W_k(N,n) = k^{-2} \sum_\rho \frac { \rho^{-n} }{ (1-\rho^1)(1-\rho^2)  \cdots (1-\rho^{N-k})   } .
\end{equation}
As we saw in Proposition \ref{well}, \eqref{ptk} must be a rational number and for fixed $k$ and $N$ it depends only on  $n \bmod k$. We next show how  these rationals may be expressed more explicitly. Note also the similarity of \eqref{ptk} with the Fourier-Dedekind sums of \cite{Be2}.

In the simplest case $k=N$ of \eqref{ptk}
\begin{equation}\label{ptkn}
    W_N(N,n) = N^{-2} \sum_\rho  \rho^{-n}
\end{equation}
which was already given in \cite[Eq. (47)]{FelRub06}.  The sum in \eqref{ptkn} is over all primitive $N$-th roots of unity. Such sums are called   Ramanujan sums and may be evaluated, see \cite[Thm. 271]{HW}, as the integer
\begin{equation*}
    c_N(n):= \sum_\rho  \rho^{n} = \sum_{d \mid (N,n)} d \cdot \mu(N/d)
\end{equation*}
where $\mu(m)$ is the M\"obius function, defined to be $0$ unless $m$ is squarefree and otherwise $-1$ to the power of the number of prime factors of $m$ (with $\mu(1)=1$).  The next result gives formulas for $ W_k(N,n)$ explicitly in terms of rationals  when $k$ is  $N,$ $N-1$ or $N-2$. Recall the Bernoulli polynomials $B_m(x)$ from \eqref{bepo} with $B_1(x)=x-1/2$ and $B_2(x)=x^2-x+1/6$.

\begin{prop} \label{ivo}
For  $n \in \Z$
\begin{align}\label{ex1}
    W_k(k,n) & = \frac 1{k^2} c_k(n) \qquad (k\in \Z_{\gqs 1}),\\
    W_{k}(k+1,n) & = \frac{-1}{k^{2}} \sum_{j=0}^{k-1}B_1\left(\frac j{k}\right)c_{k}(j-n)\qquad (k\in \Z_{\gqs 2}),\label{ex2}\\
    W_{k}(k+2,n) & = \frac{-1}{4k^{2}} \sum_{j=0}^{k-1}\left[3 B_1\left(\frac j{k}\right)+k \cdot B_2\left(\frac j{k}\right)\right]c_{k}(j-n)  - \frac{(-1)^n}{4 k^{2}} \sum_{j=0}^{\ell-1}B_1\left(\frac j{\ell}\right) c_{\ell}(j-n)  \label{ex3}
\end{align}
where in \eqref{ex3} we have $k\in \Z_{\gqs 3}$ and
\begin{equation*}
    \ell:=\begin{cases}
    k & \text{if} \quad k\equiv 0 \bmod 4,\\
    k/2 & \text{if} \quad  k\equiv 2 \bmod 4,\\
    2k & \text{if}  \quad k\equiv 1,3 \bmod 4.
    \end{cases}
\end{equation*}
\end{prop}
%define $\ell$ to be $k$ if $k\equiv 0 \bmod 4$, $k/2$ if $k\equiv 2 \bmod 4$ and $2k$ otherwise.
\begin{proof}
We have already seen \eqref{ex1}.
With $m=1$ in \eqref{psi1} and \eqref{ret} we find
\begin{equation*}
    \frac{1}{1-\rho} = -\beta_1(\rho) = -\sum_{j=0}^{k-1} B_1(j/k) \rho^j
\end{equation*}
for $\rho$ any $k$-th root of unity with $\rho\neq 1$. Summing over all primitive $k$-th root of unity we have
\begin{equation}\label{hlee}
    W_{k}(k+1,n) = \frac{1}{k^{2}} \sum_\rho \frac { \rho^{-n} }{ 1-\rho} = \frac{-1}{k^{2}} \sum_{j=0}^{k-1} B_1(j/k) \sum_\rho  \rho^{j-n}
\end{equation}
and \eqref{ex2} follows. For $k \gqs 3$ we have
\begin{equation}\label{hlee2}
    W_{k}(k+2,n) = \frac{1}{k^{2}} \sum_\rho \frac { \rho^{-n} }{ (1-\rho)(1-\rho^2)} = \frac{-1}{k^{2}} \sum_\rho \left[ \frac 14 \frac { \rho^{-n} }{ 1-\rho} + \frac 14 \frac { \rho^{-n} }{ 1+\rho} + \frac 12 \frac { \rho^{-n} }{ (1-\rho)^2} \right]
\end{equation}
using the partial fraction decomposition for $1/((1-\rho)(1-\rho^2))$ as in \cite[p. 735]{OS15}. The first sum on the right of \eqref{hlee2} has been found with \eqref{hlee}. The second sum may be evaluated in the same way since
\begin{equation*}
    \sum_\rho \frac { \rho^{-n} }{ 1+\rho} = (-1)^n \sum_\rho \frac { (-\rho)^{-n} }{ 1-(-\rho)}
\end{equation*}
and $\rho$ is a primitive $k$-th root of unity if and only if $-\rho$ is a primitive $\ell$-th root of unity. The third sum on the right of \eqref{hlee2} can be found similarly. Use $m=2$ in \eqref{ret} and \eqref{psi1} to get
\begin{equation*}
    \frac{1}{(1-\rho)^2} - \frac{1}{1-\rho} = \frac{-\beta_2(\rho)}{2} = \frac{-k}2 \sum_{j=0}^{k-1} B_2(j/k) \rho^j. \qedhere
\end{equation*}
\end{proof}
It should be possible to find similar  formulas for $W_{k}(k+3,n),$ $W_{k}(k+4,n)$ etc. When $k$ is a prime $p$, Proposition \ref{ivo} implies
\begin{align}
    W_p(p,n) & = \bigl[p-1,-1,-1,\cdots,-1 \bigr]/p^2, \label{tg1}\\
    W_p(p+1,n) & = \bigl[p',p'-1,p'-2,\cdots,p'-p+1 \bigr]/p^2 \label{tg2}
\end{align}
in the notation \eqref{wkv} and for $p':=(p-1)/2$. The identity \eqref{tg1} appears in \cite[p. 131]{Ca}. Also
\begin{equation} \label{tg3}
    W_p(p+2,n) = \frac{1}{4p^2}\left( -\overline{n}^2+\overline{n}(p-3)+ \frac{(-1)^{\overline{n}} p}2 - \frac{p^2-9p+11}{6}\right)
\end{equation}
for $p \gqs 3$, where $n\equiv \overline{n} \bmod p$ and $0\lqs \overline{n} \lqs p-1$.

\subsection{General denumerants} \label{gend}

The results  in Section \ref{swq} may be extended in a straightforward manner to the general restricted partition problem considered by Cayley and Sylvester. Let $A=\{a_1,a_2,\cdots, a_N\}$ be a fixed set of positive integers, not necessarily distinct, and write $p_A(n)$ for the number of solutions   to
\begin{equation*}
    a_1 x_1+ a_2 x_2 + \dots +a_N x_N =n \qquad (x_1, x_2, \dots , x_N \in \Z_{\gqs 0}).
\end{equation*}
Our focus, $p_N(n)$, equals $p_A(n)$ for $A=\{1,2,\cdots,N\}$. In general for $n \in \Z_{\gqs 0}$,
\begin{equation} \label{sylthm2}
    p_A(n)=\sum_{k | \text{ some }a_i} W_k(A,n)
\end{equation}
where
\begin{equation} \label{wkv2}
    W_k(A,n) = \bigl[ w_{k,0}(A,n), \ w_{k,1}(A,n), \ \ldots, \ w_{k,k-1}(A,n) \bigr],
\end{equation}
and, with a sum  over all primitive $k$-th roots of unity $\rho$ as before,
\begin{equation} \label{waveb}
    W_k(A,n)=\res_{z=0} \sum_\rho \frac{\rho^n e^{n z}}{(1-\rho^{-a_1} e^{-a_1z})(1-\rho^{-a_2} e^{-a_2z}) \cdots (1-\rho^{-a_N} e^{-a_Nz})}.
\end{equation}

\begin{prop} \label{well2}
For each wave $W_k(A,n)$, equation \eqref{wkv2} is valid for polynomials $w_{k,j}(A,x) \in \Q[x]$ that have degree at most  one less than the number of elements of $A$ that are divisible by $k$.
With any factorization $k=bc$ and $0\lqs \ell \lqs b-1$, they satisfy
\begin{equation} \label{pdf22}
    w_{k,\ell}(A,x)+ w_{k,\ell+b}(A,x)+w_{k,\ell+2b}(A,x)+\cdots + w_{k,\ell+(c-1)b}(A,x)=0 \qquad \quad (c \gqs 2).
\end{equation}
\end{prop}

%*** Fourier-Dedekind sums...
%Look at the other results in this general case....

%For $n \in \Z_{<0}$ the right side of \eqref{sylthm} makes sense and, as we see in Section \ref{sw},
%\begin{equation}\label{extp}
%    p_N(n) = \begin{cases}
%    0 & \text{ for \ \ \ }-N(N+1)/2<n<0,\\
%    (-1)^{N+1} p_N\bigl(-n-N(N+1)/2\bigr) & \text{ for \ \ \ }n \lqs -N(N+1)/2.
%    \end{cases}
%\end{equation}

\section{Asymptotics for Sylvester waves}
For $N \in \Z_{\gqs 1}$ and $n \in \Z$, Theorem \ref{qn} implies
\begin{equation} \label{slee}
    p_N(n)=-\sum_{h/k \in \farey_N} Q_{hk(-n)}(N),
\end{equation}
where we define $p_N(n)$ to be $0$ for $-N(N+1)/2<n<0$ and $(-1)^{N+1} p_N \bigl(-n-N(N+1)/2 \bigr)$ if
$n \lqs -N(N+1)/2$.
Put
\begin{equation} \label{a(n)}
    \mathcal A(N)  := \Bigl\{ h/k \ : \ N/2  <k \lqs N,  \ h=1 \text{ \ or \ } h=k-1 \Bigr\} \subseteq \farey_N
\end{equation}
and for $N$ large we partition $\farey_N$ into three parts: $\farey_{100}$, $\mathcal A(N)$ and the rest. The sum \eqref{slee} becomes
\begin{equation} \label{slee2}
    p_N(n)=-\left[\sum_{h/k \in \farey_{100}}  + \sum_{h/k \in \farey_N- (\farey_{100} \cup \mathcal A(N))} + \sum_{h/k \in \mathcal A(N)}\right] Q_{hk(-n)}(N).
\end{equation}
%When $n = \lambda N$ for $0\lqs \lambda \lqs \lambda^+$, we will see that the sums over  $\farey_{100}$ and $\mathcal A(N)$ are the largest parts on the right of \eqref{slee2}.
The sum over $\farey_{100}$ is the sum of the first $100$ waves
\begin{equation} \label{w100}
    \sum_{k=1}^{100} W_k(N,n) = -\sum_{h/k \in \farey_{100}} Q_{hk(-n)}(N).
\end{equation}
Write the sum over $\mathcal A(N)$, see \cite[Eq. (1.21)]{OS1}, as
\begin{equation} \label{sum}
    \mathcal A_1(N,-n):=\sum_{h/k \in \mathcal A(N)} Q_{hk(-n)}(N).
\end{equation}
With \eqref{slee2}, \eqref{w100} and \eqref{sum} we may write our key identity as
\begin{equation}\label{keyi}
    \sum_{k=1}^{100} W_k(N,n)=p_N(n) +\mathcal A_1(N,-n) + \sum_{h/k \in \farey_N- (\farey_{100} \cup \mathcal A(N))}  Q_{hk(-n)}(N).
\end{equation}

Since every $Q_{hk(-n)}(N)$ in \eqref{sum} is the residue of a simple pole, they may be calculated as in \cite[Eq. (1.22)]{OS1} to obtain
\begin{equation}\label{cj1}
\mathcal A_1(N,-n) = \Im \sum_{ \frac{N}{2}  <k \lqs N} \frac{2(-1)^{k}}{k^2}\exp\left( \frac{i\pi}{2}\left[ \frac{-N^2-N-4 n}{k}+3N \right]\right) \spr{1/k}{N-k}
\end{equation}
where we used the reciprocal of the product
\begin{equation} \label{sidef}
    \spn{\theta}{m} :=\prod_{j=1}^m 2\sin (\pi j \theta)
\end{equation}
with $\spn{\theta}{0}:=1$. The right side of \eqref{cj1} can be analyzed in great detail and its asymptotics found:
%In fact we may replace $n$ on the right of \eqref{cj1} by any real number.
%The equality \eqref{cj1} is valid only for $\sigma \in \Z$. We may redefine $\mathcal A_1(N,\sigma)$ for all $\sigma \in \R$ with \eqref{cj1}.

\begin{theorem}\label{thma} Let $\lambda^+$ be a positive real number. Suppose $N \in \Z_{\gqs 1}$ and  $\lambda N \in \Z$ for $\lambda$ satisfying $|\lambda| \lqs \lambda^+$.
Then for $a_{0}(\lambda)=2  z_0 e^{-\pi i z_0(1+2\lambda)}$ and  explicit  $a_{1}(\lambda),$ $a_{2}(\lambda), \dots $  we have
\begin{equation} \label{thmaa}
   \mathcal A_1(N,-\lambda N) = \Re\left[\frac{w_0^{-N}}{N^{2}} \left( a_{0}(\lambda)+\frac{a_{1}(\lambda)}{N}+ \dots +\frac{a_{m-1}(\lambda)}{N^{m-1}}\right)\right] + O\left(\frac{|w_0|^{-N}}{N^{m+2}}\right)
\end{equation}
for an implied constant depending only on  $\lambda^+$ and $m$.
\end{theorem}

This theorem is proved in Section \ref{klmx}. The next result, on the last component of \eqref{keyi}, is shown in Section \ref{klmx2}.

\begin{theorem}\label{thmb} Let $\lambda^+$ be a positive real number. Suppose $N \in \Z_{\gqs 1}$ and  $\lambda N \in \Z$ for $\lambda$ satisfying $|\lambda| \lqs \lambda^+$.
Then for an implied constant depending only on $\lambda^+$,
\begin{equation} \label{indx}
   \sum_{h/k \in \farey_N-(\farey_{100} \cup \mathcal A(N))} Q_{hk(-\lambda N)}(N) = O\left(e^{0.055N}\right).
\end{equation}
\end{theorem}

 These two results, combined with bounds for $p_N(n)$, imply Theorem \ref{ma2}:

\begin{proof}[Proof of Theorem \ref{ma2}]
Combining \eqref{keyi} with Theorems \ref{thma}, \ref{thmb} shows
\begin{align}
    \sum_{k=1}^{100} W_k(N,\lambda N) & = p_N(\lambda N) +  \mathcal A_1(N,-\lambda N) + O\left(e^{0.055 N}\right) \notag\\
     & = p_N(\lambda N) + \Re\left[\frac{w_0^{-N}}{N^{2}} \left( a_{0}(\lambda)+\frac{a_{1}(\lambda)}{N}+ \dots +\frac{a_{m-1}(\lambda)}{N^{m-1}}\right)\right] +  O\left(\frac{|w_0|^{-N}}{N^{m+2}}\right) \label{ptg}
\end{align}
where $|w_0|^{-N} \approx e^{0.068N}$. The estimate
\begin{equation} \label{ptg2}
 p_N(\lambda N) \ll  e^{\bigl(2\pi \sqrt{|\lambda|/6}\bigr)\sqrt{N}}
\end{equation}
follows from $|p_N(n)|\lqs p_N(|n|) \lqs p(|n|)$ and \eqref{pnexp}.
Calculus shows the crude bound $e^{r\sqrt{N}}\lqs e^{r^2/(4t)}\cdot e^{tN}$ for any $r\gqs 0,$ $t>0$. Hence, with $r=2\pi \sqrt{|\lambda|/6}$ and $t=0.055$, say,
\begin{equation*}
 p_N(\lambda N) \ll  e^{\bigl(2\pi \sqrt{|\lambda|/6}\bigr)\sqrt{N}} \lqs e^{\pi^2|\lambda|/(6(0.055))} e^{0.055N}.
\end{equation*}
Therefore $p_N(\lambda N)$ may be included in the error term in \eqref{ptg}. This completes the proof of Theorem \ref{ma2}.
\end{proof}

\section{Required results}
\subsection{The saddle-point method}

We will apply Perron's saddle-point method from \cite{Pe17} in Sections \ref{klmx} and \ref{sec7}. The exact form we need is given in \cite{OSper} and requires the following discussion to state it precisely.

The usual convention that the principal branch of $\log$ has arguments in $(-\pi,\pi]$ is used.
As in  \eqref{as} below, powers of nonzero complex numbers take the corresponding principal value $z^{\tau}:=e^{\tau\log(z)}$
for $\tau \in \C$.

Our contours of integration $\cc$ will lie in a bounded region of $\C$ and be parameterized by
  a continuous function   $c:[0,1]\to \C$ that has a continuous derivative except at a finite number of points.  For any appropriate $f$, integration along the corresponding contour $\cc$ is defined as $\int_\cc f(z)\, dz := \int_0^1 f(c(t))c'(t)\, dt$ in the normal way. %To avoid trivial cases, it is assumed that $c(0) \neq c(1)$. We  may also assume that the contour has a specific angle at certain points. For example, if  $$\lim_{t\to 0^+} \arg(c(t)-c(0))$$ exists, it gives  the angle that $\cc$ leaves $c(0)$.

We make the following assumptions and definitions.

\begin{assume} \label{asma}
We have $\nb$ a neighborhood of $z_0 \in \C$.  Let $\cc$ be  a contour as described above, with $z_0$  a point on it. Suppose $p(z)$ and $q(z)$ are holomorphic functions on a domain containing $\nb \cup \cc$.  We  assume $p(z)$ is not constant and  hence there must exist $\mu \in \Z_{\gqs 1}$ and $p_0 \in \C_{\neq 0}$ so that
\begin{equation}
    p(z)  =p(z_0)+p_0(z-z_0)^\mu(1-\phi(z)) \qquad  (z\in \nb) \label{funp}
\end{equation}
with $\phi$  holomorphic on $\nb$ and $\phi(z_0)=0$.
Let $\omega_0:=\arg(p_0)$ and we will need the {\em steepest-descent angles}
\begin{equation}\label{bisec}
    \theta_\ell := -\frac{\omega_0}{\mu}+\frac{2\pi \ell}{\mu} \qquad (\ell \in \Z).
\end{equation}
We also assume that  $\nb,$ $\cc,$ $p(z),$  $q(z)$ and  $z_0$  are independent of  $N>0$. Finally, let $K(q)$ be a bound for $|q(z)|$ on $\nb \cup \cc$.
\end{assume}

\begin{theorem}[The saddle-point method of Perron] \label{sdle}
Suppose Assumptions \ref{asma} hold and $\mu$ is even. Let $\cc$ be a contour beginning at $z_1$, passing through $z_0$ and ending at $z_2$, with these points all distinct.  Suppose that
\begin{equation*}
   \Re(p(z))>\Re(p(z_0)) \quad \text{for all} \quad z \in \cc, \ z\neq z_0.
\end{equation*}
Let $\cc$ approach $z_0$ in a sector of angular width $2\pi/\mu$ about $z_0$ with bisecting angle $\theta_{k}\pm\pi$, and  initially leave $z_0$ in a sector of the same size with bisecting angle $\theta_{k}$.
 Then for every $M \in \Z_{\gqs 0}$,
\begin{equation*} %\label{ilm}
    \int_\cc e^{-N \cdot p(z)} q(z) \, dz = e^{-N \cdot p(z_0)} \left(\sum_{m=0}^{M-1}  \G\left(\frac{2m+1}{\mu}\right) \frac{2\alpha_{2m}(q)   \cdot e^{2\pi i k (2m+1)/\mu}}{N^{(2m+1)/\mu}} + O\left(\frac{K(q)}{N^{(2M+1)/\mu}} \right)  \right)
\end{equation*}
as $N \to \infty$ where the implied constant  is independent of $N$ and $q$. The numbers $\alpha_{s}(q)$ are given by
\begin{equation} \label{as}
    \alpha_s(q) = \frac{1}{\mu \cdot  s!} p_0^{-(s+1)/\mu} \frac{d^s}{dz^s}\left\{q(z) \cdot \left( 1-\phi(z)\right)^{-(s+1)/\mu} \right\}_{z=z_0}.
\end{equation}.
\end{theorem}

Theorem \ref{sdle} is proved as Corollary 5.1 in \cite{OSper} with the innovation of making the error independent of $q$. Note that \cite{OSper} has $p(z)$ with the opposite sign but all other notation, e.g. $p_0$, is the same.
The numbers $\alpha_{s}(q)$ depend on $p$ and $z_0$, but we have highlighted the dependence  on $q$ since we will be applying Theorem \ref{sdle} with $q$ varying.
Another description of $\alpha_{s}(q)$ may be given in terms of the power series for $p$ and $q$ near $z_0$:
\begin{equation} \label{pspq}
    p(z)-p(z_0)=\sum_{s=0}^\infty p_s (z-z_0)^{s+\mu}, \qquad q(z)=\sum_{s=0}^\infty q_s (z-z_0)^{s}.
\end{equation}
This requires
 the {\em partial ordinary Bell polynomials}, \cite[p. 136]{Comtet}, which are defined with the generating function
\begin{equation} \label{pobell2}
    \left( p_1 x +p_2 x^2+ p_3 x^3+ \cdots \right)^j = \sum_{i=j}^\infty \hat{B}_{i,j}(p_1, p_2, p_3, \dots) x^i.
\end{equation}
Clearly $\hat{B}_{i,0}(p_1, p_2, p_3, \dots)$ is $1$ for $i=0$ and is $0$ for $i \gqs 1$. Also
\begin{equation} \label{pobell3}
    \hat{B}_{i,j}(p_1, p_2, p_3, \dots)= \sum_{n_1+n_2+\dots + n_j = i}
    p_{n_1}p_{n_2} \cdots p_{n_j}
\end{equation}
for $j \gqs 1$ from \cite[p. 156]{CFW} where the sum is over all possible $n_1$, $n_2,  \dots \in \Z_{\gqs 1}$.

\begin{prop} \label{wojf}
For $\alpha_s(q)$ defined in \eqref{as},
\begin{equation} \label{hjw}
    \alpha_s(q) = \frac{1}{\mu} p_0^{-(s+1)/\mu} \sum_{i=0}^s q_{s-i}\sum_{j=0}^i  \binom{-(s+1)/\mu}{j}  \hat B_{i,j}\left(\frac{p_1}{p_0},\frac{p_2}{p_0},\cdots\right).
\end{equation}
\end{prop}

Proposition \ref{wojf} is due to Campbell,  Fr{\"o}man and Walles \cite[pp. 156--158]{CFW}. The above formulation is proved in \cite[Prop. 7.2]{OSper}. Wojdylo also rediscovered this formula; see the references in \cite{OSper}. The following result is \cite[Prop. 7.3]{OSper}.

\begin{prop} \label{albnd}
With  Assumptions \ref{asma} and $\alpha_s(q)$ defined in \eqref{as},
\begin{equation} \label{albnd2}
    \alpha_{s}(q) = O\bigl( K^*(q) \cdot C^s \bigr) \quad \text{for} \quad s \in \Z_{\gqs 0}
\end{equation}
where $K^*(q)$ is a bound for  $|q(z)|$ on $\nb$. The positive constant $C$ and the implied constant in \eqref{albnd2} are both independent of $q$ and $s$.
\end{prop}

\subsection{The dilogarithm} \label{dise}
As described in  \cite{max}, \cite{Zag07}, \cite{OS3} for example, the dilogarithm is
initially defined as
\begin{equation}\label{def0}
\li(z):=\sum_{n=1}^\infty \frac{z^n}{n^2} \quad \text{ for }|z|\lqs 1,
\end{equation}
with an analytic continuation given by
$
 -\int_{0}^z \log(1-u)/u \, du
$. This makes the dilogarithm a multi-valued holomorphic function with a branch points at $1,$ $\infty$ (and off the principal branch another branch point at $0$).
 We let $\li(z)$ denote the dilogarithm on its principal branch so that $\li(z)$ is a single-valued holomorphic  function on $\C-[1,\infty)$.

We may describe $\li(z)$ for $z$ on the unit circle as
\begin{alignat}{2}
    \Re(\li(e^{2\pi i x}) ) & = \sum_{n=1}^\infty \frac{\cos(2\pi n x)}{n^2} = \pi^2 B_2(x-\lfloor x \rfloor) \qquad & & (x\in \R), \label{reli}\\
    \Im(\li(e^{2\pi i x}) ) & = \sum_{n=1}^\infty \frac{\sin(2\pi n x)}{n^2} = \cl(2\pi x) \qquad & & (x\in \R) \label{imc}
\end{alignat}
where $B_2(x):=x^2-x+1/6$ is the second  Bernoulli polynomial and
\begin{equation}\label{simo}
\cl(\theta):=-\int_0^\theta \log |2\sin( x/2) | \, dx \qquad (\theta \in \R)
\end{equation}
is Clausen's integral.
Note that $\li(1)=\zeta(2)=\pi^2/6$.
The graph of  $\cl(\theta)$ resembles a slanted sine wave - see \cite[Fig. 1]{OS3}, for example. Combine \eqref{reli} and \eqref{imc} to get, for $z\in \R$ with $m\lqs z \lqs m+1, m\in \Z$,
\begin{equation}\label{dli}
\cl\left(2\pi  z\right) = -i \li\left( e^{2\pi i z}\right)+i\pi^2\left(z^2-(2m+1)z+m^2+m+1/6 \right).
\end{equation}
Then  the right of \eqref{dli} gives the continuation of $\cl\left(2\pi  z\right)$ to $z \in \C$ with $m<\Re(z)<m+1$.

  As $z$ crosses the branch cuts  the dilogarithm enters new branches. From \cite[Sect.  3]{max}, the value of the analytically continued dilogarithm is always given by
\begin{equation}\label{dilogcont}
\li(z) + 4\pi^2  A +   2\pi i  B  \log \left(z\right)
\end{equation}
for some $A$, $B \in \Z$.

The saddle-points we need in our asymptotic calculations are closely related to zeros of the analytically continued dilogarithm and in \cite{OS3} we have made a study of  its zeros on every branch. When the continued dilogarithm takes the form \eqref{dilogcont} with $B=0$, there will be a  zero if and only if $A\gqs 0$ and, for each such $A$, the zero will be unique and lie on the real line. The cases we will require have $B\neq 0$. In these cases there are no real zeros so we may avoid the branch cuts and look for solutions to
\begin{equation}\label{dilogzero}
\li(z) + 4\pi^2  A +   2\pi i  B  \log \left(z\right)=0 \qquad (z\in \C, z \not\in (-\infty,0] \cup [1,\infty), \ A,B \in \Z).
\end{equation}
The next result is shown in Theorems 1.1 and 1.3 of \cite{OS3}.

\begin{theorem} \label{dilab}
For nonzero $B \in \Z$, \eqref{dilogzero} has solutions  if and only if $-|B|/2<A\lqs |B|/2$. For such a pair $A,B$ the solution $z$ is unique.
This unique solution, $w(A,B)$,  may be found to arbitrary precision using Newton's method.
\end{theorem}

 By conjugating \eqref{dilogzero} it is clear that
\begin{equation*}
    w(A,-B) = \overline{w(A,B)}.
\end{equation*}
So for nonzero $B$ the first zeros are $w(0,1)$ and its conjugate $w(0,-1)$.   We have
\begin{equation*}\label{w0-1}
    w(0,-1) \approx \phantom{-}0.9161978162 - 0.1824588972 i
\end{equation*}
and this zero was denoted by $w_0$ in Section \ref{maru}.
We will also need
\begin{align*} %\label{w0-2}
    w(0,-2) &\approx \phantom{-}0.9684820460-0.1095311065 i, \\
    w(1,-3) &\approx -0.4594734813-0.8485350380 i. %\label{w1-3}
\end{align*}

Define
\begin{equation}\label{pdfn}
    p_d(z):=\frac{ - \li\left(e^{2\pi i z}\right) +\li(1) +4\pi^2 d}{2\pi i z},
\end{equation}
a single-valued holomorphic function away from the vertical branch cuts $(-i\infty,n]$ for $n \in \Z$. Its first derivatives are
\begin{align}\label{pzzz}
    p_d'(z) & =-\frac 1z \left(p_d(z)-\log \left(1-e^{2\pi i z}\right) \right),\\
    p_d''(z) & =-\frac 1z \left(2 p_d'(z) + \frac{2\pi i \cdot e^{2\pi i z}}{1-e^{2\pi i z}} \right). \label{pzzz2}
\end{align}
The last result in this section is \cite[Thm. 2.4]{OS1} and identifies the saddle-points we will need.

\begin{theorem} \label{disol}
Fix integers $m$ and $d$ with $-|m|/2<d\lqs |m|/2$. Then there is a unique solution to $p_d'(z)=0$ for $z \in \C$ with $m-1/2<\Re(z)<m+1/2$ and $z \not\in (-i\infty,m]$. Denoting this solution by $z^*$, it is given by
\begin{equation}\label{uniq}
    z^*=m+ \log \bigl(1-w(d,-m)\bigr)/(2\pi i)
\end{equation}
and satisfies
\begin{equation} \label{pzlogw}
    p_d(z^*)=\log \bigl(w(d,-m)\bigr).
\end{equation}
\end{theorem}

\section{Proof of Theorem \ref{thma}} \label{klmx}

\begin{proof}[Proof of Theorem \ref{thma}] We require a modification of the proof of \cite[Thm 1.6]{OS1}.
The following should be read alongside Sections 2-5 of \cite{OS1} where there are more details.
First define
\begin{equation}\label{grz}
g_\ell(z):=-\frac{B_{2\ell}}{(2\ell)!} \left( \pi z\right)^{2\ell-1} \cot^{(2\ell-2)}\left(\pi z\right).
\end{equation}
We approximate the reciprocal of the sine product $\spn{\theta}{m}$, as defined in \eqref{sidef}, with \cite[Thm. 4.1]{OS1}, based on Euler-Maclaurin summation. A special case, see \cite[Cor. 4.2]{OS1}, shows
\begin{equation}\label{ott1}
    \spr{1/k}{N-k} = O\left(e^{0.05N}\right) \quad \text{ for } \quad  \frac Nk \in [1,1.01] \cup  [1.49 ,2)
\end{equation}
and
\begin{multline} \label{ott2}
    \spr{1/k}{N-k} =  \left(\frac{N/k}{2N \sin (\pi (N/k-1))}\right)^{1/2}  \exp\left(\frac N{2\pi N/k} \cl\bigl(2\pi  N/k \bigr) \right) \\
      \quad \times
      \exp\left(\sum_{\ell=1}^{L-1} \frac{g_{\ell}(N/k)}{N^{2\ell-1}} \right) +O\left(e^{0.05N}\right)
        \quad \text{ for } \quad \frac Nk \in  (1.01, 1.49)
  \end{multline}
with $L=\lfloor 0.006\pi e \cdot N\rfloor$. The implied constants in \eqref{ott1}, \eqref{ott2} are absolute. We may combine \eqref{cj1}, \eqref{ott1} and \eqref{ott2} by first making the following definitions:
\begin{align}
p(z) & :=\frac{\li(1) - \li\left(e^{2\pi i z}\right)}{2\pi i z},\label{p(z)} \\
q(z) & := \left( \frac{z }{2\sin(\pi(z -1))}\right)^{1/2} \exp(-\pi i z/2), \label{qz}\\
v(z;N,\sigma)  & := \frac{2\pi i \sigma z}{N} + \sum_{\ell=1}^{L-1} \frac{g_{\ell}(z)}{N^{2\ell-1}}, \qquad \quad (L=\lfloor 0.006\pi e \cdot N \rfloor). \label{vz}
\end{align}
We used \eqref{dli} with $m=1$ to convert Clausen's integral into the dilogarithm in \eqref{p(z)}. (Note that all the occurrences of the function $p$ in this section refer to \eqref{p(z)} and not the partition function.) Then with $\hat{z}=\hat{z}(N,k):=N/k$  define
\begin{equation} \label{a2ns}
    \mathcal A_2(N,\sigma) := \frac 2{N^{1/2}} \Im  \sum_{ k \ : \ \hat{z} \in  (1.01, 1.49)} \frac{(-1)^{k}}{k^2}  \exp\bigl(-N  (p(\hat{z})-\pi i/\hat{z}) \bigr) q\left(\hat{z} \right) \exp\bigl(v \left(\hat{z};N,\sigma \right) \bigr)
\end{equation}
where the index notation means we are summing over all integers $k$ such that $1.01<N/k<1.49$.
It follows, as in \cite[Eq. (4.11)]{OS1}, that for any $\sigma \in \Z$ and an absolute implied constant we have
\begin{equation}\label{a1n}
\mathcal A_2(N,\sigma) = \mathcal A_1(N,\sigma) +O(e^{0.05N}).
\end{equation}
 When $\sigma = -\lambda N$ we may remove the $\exp(-2\pi i \lambda z)$ factor from $\exp\bigl(v \left(z;N,\sigma \right) \bigr)$ and include  it in the new function
\begin{equation}
    f_{\mathcal A,\lambda}(z)  := q(z) \exp(-2\pi i \lambda z ) \label{rz}
\end{equation}
to get
\begin{equation} \label{incheg}
    \mathcal A_2(N,-\lambda N) = \frac 2{N^{1/2}} \Im  \sum_{ k \ : \ \hat{z} \in  (1.01, 1.49)} \frac{(-1)^{k}}{k^2}  \exp\bigl(-N  (p(\hat{z})-\pi i/\hat{z}) \bigr) f_{\mathcal A,\lambda}(\hat{z}) \exp\bigl(v \left(\hat{z};N,0 \right) \bigr).
\end{equation}
It is shown in \cite[Prop. 4.7]{OS1} that $q(z)$ and $\exp\bigl(v \left(z;N,0 \right)\bigr)$ are holomorphic and absolutely bounded on a domain containing the box $\mathbb B_1$ for
\begin{equation}\label{box}
    \mathbb B_m := \{z\in \C \ : \ m+0.01\lqs \Re(z) \lqs m+0.49, \ -1 \lqs \Im(z) \lqs 1\}.
\end{equation}
Since $|\exp(-2\pi i \lambda z )| \lqs \exp(\lambda^+ 2\pi   |z| )$ it follows that
\begin{equation}\label{soota}
    f_{\mathcal A,\lambda}(z)  \exp\bigl(v \left(z;N,0 \right)\bigr) \ll 1 \quad \text{for} \quad z \in \mathbb B_1
\end{equation}
with an implied constant depending only on $\lambda^+$.
The proof of \cite[Thm. 4.3]{OS1} now goes through unchanged, except that we have  $\sigma=0$ and $f_{\mathcal A,\lambda}$ instead of $q$. This theorem allows us
to replace the  sum \eqref{incheg} by the integral
\begin{equation}
    \mathcal A_3(N,-\lambda N)  :=  \frac{2}{N^{3/2}} \Im \int_{1.01}^{1.49} e^{-N \cdot p(z)}  f_{\mathcal A,\lambda}(z) \cdot \exp\bigl(v(z;N,0)\bigr) \, dz \label{a3(n)}
\end{equation}
where, for an implied constant depending only on  $\lambda^+$ (coming from the bound \eqref{soota}),
\begin{equation}\label{a23n}
\mathcal A_3(N,-\lambda N) = \mathcal A_2(N,-\lambda N) +O(e^{0.05N}).
\end{equation}

%*************************************
% Graphics   paths p,q,s

\SpecialCoor
\psset{griddots=5,subgriddiv=0,gridlabels=0pt}
\psset{xunit=1cm, yunit=0.6cm}
\psset{linewidth=1pt}
\psset{dotsize=4pt 0,dotstyle=*}

\newrgbcolor{darkbrn}{0.7 0.2 0.7}

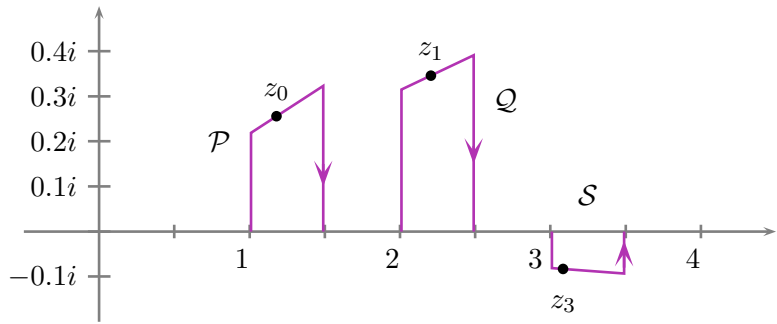
\begin{figure}[ht]
\begin{center}
\begin{pspicture}(-1,-2)(9,5) %\psgrid

\psline[linecolor=gray]{->}(0,-2)(0,5)
\psline[linecolor=gray]{->}(-1,0)(9,0)

\multirput(-0.15,-1)(0,1){6}{\psline[linecolor=gray](0,0)(0.3,0)}
\multirput(1,-0.15)(1,0){8}{\psline[linecolor=gray](0,0)(0,0.3)}

\psset{arrowscale=2,arrowinset=0.5}
%\psline[linecolor=darkbrn]{->}(2.02,1)(2.02,1.5)
\psline[linecolor=darkbrn]{->}(2.98,1.5)(2.98,1)
\psline[linecolor=darkbrn](2.02,0)(2.02,2.19)(2.98,3.23)(2.98,0)
\psdots(2.36,2.56)
\rput(2.36,3.1){$z_0$}

%\psline[linecolor=darkbrn]{->}(4.02,1)(4.02,1.5)
\psline[linecolor=darkbrn]{->}(4.98,1.9)(4.98,1.5)
\psline[linecolor=darkbrn](4.02,0)(4.02,3.15)(4.98,3.91)(4.98,0)
\psdots(4.41,3.46)
\rput(4.41,4){$z_1$}

\psline[linecolor=darkbrn](6.02,0)(6.02,-0.81)(6.98,-0.93)(6.98,0)
\psline[linecolor=darkbrn]{->}(6.98,-0.8)(6.98,-0.2)
\psdots(6.17,-0.83)
\rput(6.17,-1.6){$z_3$}

\rput(1.9,-0.6){$1$}
  \rput(3.9,-0.6){$2$}
  \rput(5.8,-0.6){$3$}
  \rput(7.9,-0.6){$4$}
  \rput(1.6,2){$\mathcal P$}
  \rput(5.4,2.9){$\mathcal Q$}
  \rput(6.5,0.8){$\mathcal S$}

\rput(-0.6,1){$0.1i$}
\rput(-0.6,2){$0.2i$}
\rput(-0.6,3){$0.3i$}
\rput(-0.6,4){$0.4i$}
\rput(-0.75,-1){$-0.1i$}

\end{pspicture}
\caption{The paths of integration $\mathcal P$, $\mathcal Q$ and $\mathcal S$}\label{pthpqs}
\end{center}
\end{figure}
%
%
%*************************************

The form of \eqref{a3(n)} allows us to find its asymptotic expansion using the saddle-point method as was done in \cite{OS1}.
We have seen in Theorem \ref{disol} that $p'(z)=0$ has a unique solution for $0.5<\Re(z)<1.5$ given by $z=1+\log \bigl(1-w(0,-1)\bigr)/(2\pi i)$. (The function $p(z)$ is the case $d=0$ of $p_d(z)$.) In the notation of \eqref{w0x} we write $w_0=w(0,-1)$ and $z_0$ is the saddle-point $1+\log(1-w_0)/(2\pi i)$.

In the notation of Assumptions \ref{asma} and Theorem \ref{sdle}, we find $\mu=2$, $p_0\approx 0.504-0.241i$ and the steepest-descent angles are $\theta_0\approx 0.223$ and $\theta_1=\pi+\theta_0$.
Let $c:=1+i\Im(z_0)/\Re(z_0)$. We move the path of integration in \eqref{a3(n)} to the path $\mathcal P$ consisting of the straight line segments joining the points $1.01,$  $1.01c,$  $1.49c$ and $1.49$. This path passes through $z_0$ as shown in Figure \ref{pthpqs}. Since the integrand in \eqref{a3(n)} is holomorphic on a domain containing $\mathbb B_1$, Cauchy's theorem ensures that the integral remains the same under this change of path.
 It is proved in \cite[Thm. 5.2]{OS1} that
\begin{equation}\label{aqw}
    \Re(p(z)-p(z_0))>0 \quad \text{for all} \quad z \in \mathcal P, \ z \neq z_0.
\end{equation}
Recall from \eqref{abuv} and \eqref{pzlogw} that
\begin{equation} \label{adzx}
 e^{-p(z_0)} = w_0^{-1} \quad \text{and} \quad   e^{-\Re(p(z_0))} = |w_0|^{-1} \approx e^{0.068}.
\end{equation}
To apply the saddle-point method we state one further result, which is \cite[Prop. 5.8]{OS1}.
Set $u_{0,0}:=1$ and for $j \in \Z_{\gqs 1}$ put
\begin{equation} \label{uiz}
    u_{0,j}(z):=\sum_{m_1+3m_2+5m_3+ \dots =j}\frac{g_1(z)^{m_1}}{m_1!}\frac{g_2(z)^{m_2}}{m_2!} \cdots \frac{g_j(z)^{m_j}}{m_j!}.
\end{equation}

\begin{prop} \label{gas}
There are  functions $u_{0,j}(z)$ (defined above) and $\zeta_d(z;N,0)$ which are holomorphic on a domain containing the box $\mathbb B_1$ and have the following property. For all $z \in \mathbb B_1$,
\begin{equation*}
    \exp\bigl(v(z;N,0)\bigr) = \sum_{j=0}^{d-1} \frac{u_{0,j}(z)}{N^j} + \zeta_d(z;N,0) \quad \text{for} \quad \zeta_d(z;N,0) = O\left(\frac{1}{N^d} \right)
\end{equation*}
with an implied constant depending only on  $d$ where $1 \lqs d \lqs 2L-1$ and $L=\lfloor 0.006 \pi e \cdot N \rfloor$.
\end{prop}

%The proof of \cite[Thm. 1.6]{OS1} now establishes...::

Proposition \ref{gas} implies
\begin{multline}\label{umand}
    \mathcal A_3(N,-\lambda N)  = \Im\Biggl[ \sum_{j=0}^{d-1} \frac{2}{N^{3/2+j}}  \int_{\mathcal P} e^{-N \cdot p(z)} \cdot f_{\mathcal A,\lambda}(z) \cdot u_{0,j}(z) \, dz\\
     + \frac{2}{N^{3/2}}  \int_{\mathcal P} e^{-N \cdot p(z)} \cdot f_{\mathcal A,\lambda}(z) \cdot \zeta_d(z;N,0) \, dz \Biggr]+ O(e^{0.06N})
\end{multline}
where, by \eqref{soota}, \eqref{aqw} and Proposition \ref{gas}, the last term in the parentheses in \eqref{umand} is
\begin{equation*}
    \ll \frac{1}{N^{3/2}}  \int_{\mathcal P} \left|e^{-N \cdot p(z)}\right| \cdot 1 \cdot \frac{1}{N^d} \, dz
    \ll \frac{1}{N^{d+3/2}} e^{-N \Re(p(z_0))} = \frac{|w_0|^{-N}}{N^{d+3/2}},
\end{equation*}
for an implied constant depending only on $\lambda^+$.
Applying Theorem \ref{sdle} to each integral  in the first part of \eqref{umand}
 we obtain, since $k=0$,
\begin{multline} \label{wmand}
    \int_{\mathcal P} e^{-N \cdot p(z)} \cdot f_{\mathcal A,\lambda}(z) \cdot u_{0,j}(z) \, dz \\
    =  e^{-N \cdot p(z_0)}\left(\sum_{m=0}^{M-1}\G\left(m+\frac 12 \right) \frac{2\alpha_{2m}(f_{\mathcal A,\lambda} \cdot u_{0,j})}{N^{m+1/2}}+O\left( \frac{K(f_{\mathcal A,\lambda} \cdot u_{0,j})}{N^{M+1/2}}\right) \right).
\end{multline}
The error term in \eqref{wmand} corresponds to an error in \eqref{umand} of size $O(|w_0|^{-N}/N^{M+j+2})$.
Choose $M=d$ so that this error  is less than $O(|w_0|^{-N}/N^{d+3/2})$ for all $j \gqs 0$.
Therefore
\begin{multline} \label{jmp}
    \mathcal A_3(N,-\lambda N)  = \Im \left[
    \sum_{j=0}^{d-1} \frac{4}{N^{j+3/2}}   e^{-N \cdot p(z_0)} \sum_{m=0}^{d-1} \G\left(m+\frac 12 \right) \frac{\alpha_{2m}(f_{\mathcal A,\lambda} \cdot u_{0,j})}{N^{m+1/2}}
    \right]+ O\left( \frac{|w_0|^{-N}}{N^{d+3/2}}\right)\\
     = \Im \left[  w_0^{-N}
    \sum_{t=0}^{2d-2} \frac{4}{N^{t+2}}    \sum_{m=\max(0,t-d+1)}^{\min(t,d-1)} \G\left(m+\frac 12 \right)  \alpha_{2m}(f_{\mathcal A,\lambda} \cdot u_{0,t-m})
    \right]+ O\left( \frac{|w_0|^{-N}}{N^{d+3/2}}\right) \\
     = \Re \left[  w_0^{-N}
    \sum_{t=0}^{d-2} \frac{-4i}{N^{t+2}}    \sum_{m=0}^{t} \G\left(m+\frac 12 \right)  \alpha_{2m}(f_{\mathcal A,\lambda} \cdot u_{0, t-m})
    \right]+ O\left( \frac{|w_0|^{-N}}{N^{d+1}}\right)
\end{multline}
for implied constants depending only on $\lambda^+$ and $d$. (In going from the previous line to \eqref{jmp} we used that $|\alpha_{2m}(f_{\mathcal A,\lambda} \cdot u_{0,j})|$ has a bound  depending only  on  $\lambda^+$ and $d$, by Proposition \ref{albnd}, when $m,$ $j \lqs d-1$.)
Recall that $\mathcal A_1(N,-\lambda N) = \mathcal A_3(N,-\lambda N) +O(e^{0.05N})$ by \eqref{a1n} and  \eqref{a23n}.
 Hence,  with
\begin{equation} \label{btys}
    a_t(\lambda):=  -4i  \sum_{m=0}^t \G\left(m+\frac 12 \right) \alpha_{2m}(f_{\mathcal A,\lambda} \cdot u_{0,t-m}),
\end{equation}
we obtain \eqref{pres} in the statement of the theorem.

The first coefficient is
\begin{equation} \label{oad}
    a_0(\lambda)=-4i \G(1/2) \alpha_0(f_{\mathcal A,\lambda} \cdot u_{0,0})= -4i \sqrt{\pi} \alpha_0(f_{\mathcal A,\lambda})= -2i \sqrt{\pi} p_0^{1/2} f_{\mathcal A,\lambda}(z_0),
\end{equation}
using \eqref{as}.
The terms $p_0$ and $q_0$ are defined in \eqref{pspq} so that, using \eqref{pzzz2} and \eqref{qz},
\begin{equation}  \label{woad}
    p_0  =p''(z_0)/2 = \frac{-\pi i e^{2\pi i z_0}}{z_0 w_0},\qquad
    q^2_0  =q(z_0)^2 = \frac{i z_0}{w_0}.
\end{equation}
Taking square roots (and numerically checking whether the sign should be $+$ or $-$),
\begin{equation} \label{woad2}
    p_0^{1/2}  = -\frac{\sqrt{\pi} e^{-\pi i/4} e^{\pi i z_0}}{z_0^{1/2} w_0^{1/2}},\qquad
    q_0   = -\frac{e^{\pi i/4} z_0^{1/2}}{w_0^{1/2}}.
\end{equation}
By \eqref{rz} and \eqref{pspq}, we have the power series
\begin{align}
    f_{\mathcal A,\lambda}(z)  & =f_0+f_1(z-z_0)+f_2(z-z_0)^2+ \cdots \notag\\
     & = e^{-2\pi i \lambda z_0}\bigl(q_0+q_1(z-z_0)+q_2(z-z_0)^2+ \cdots\bigr) \notag\\
     & \qquad\times \bigl(1-2\pi i \lambda(z-z_0)-2\pi^2 \lambda^2(z-z_0)^2+ \cdots \bigr). \label{r00}
\end{align}
This shows that
\begin{equation} \label{a0q}
   f_{\mathcal A,\lambda}(z_0) = f_0 = e^{-2\pi i \lambda z_0} q_0.
\end{equation}
Assembling \eqref{oad}, \eqref{woad2} and \eqref{a0q} proves that $a_0(\lambda)=2  z_0 e^{-\pi i z_0(1+2\lambda)}$. This completes the proof of Theorem \ref{thma}.
\end{proof}

\begin{prop} \label{propb1s}   We have
\begin{equation} \label{b1}
a_1(\lambda)= - \frac{ w_0}{ \pi i e^{\pi i z_0(3+2\lambda)}}
\left( \frac{(2\pi i z_0)^2}{12}(6\lambda^2+6\lambda+1)-2\pi i z_0(2\lambda+1) +1  \right).
\end{equation}
\end{prop}
\begin{proof} Formula \eqref{btys} implies
\begin{align*}
    a_1(\lambda) & =  -4i  \bigl(\G(1/2) \alpha_{0}(f_{\mathcal A,\lambda} \cdot u_{0,1}) + \G(3/2)\alpha_{2}(f_{\mathcal A,\lambda} \cdot u_{0,0})\bigr)\\
     & =  -4i\sqrt{\pi} \bigl( \alpha_{0}(f_{\mathcal A,\lambda}) \cdot u_{0,1}(z_0) + \alpha_{2}(f_{\mathcal A,\lambda})/2\bigr).
\end{align*}
As in \cite[Prop. 5.10]{OS1}, $u_{0,1}(z_0) = \pi i z_0( -1/2+1/w_0)/6$. Also \eqref{hjw} implies
\begin{equation*}
    \alpha_2(f_{\mathcal A,\lambda})=\frac{1}{2p_0^{1/2}}\frac{f_0}{p_0}\left(\frac{f_2}{f_0}-\frac{3}{2} \frac{p_1}{p_0} \frac{f_1}{f_0} -\frac{3}{2} \frac{p_2}{p_0} +\frac{15}{8} \frac{p_1^2}{p_0^2}\right).
\end{equation*}
From \eqref{pzzz}, \eqref{pzzz2} and their generalizations we have
\begin{equation*}
    \frac{p_1}{p_0}  = -\frac{1}{z_0}+\frac{2\pi i}{3w_0}, \qquad
    \frac{p_2}{p_0}  = \frac{\pi^2}{3w_0}  +\frac{1}{z_0^2}-\frac{2 \pi i}{3z_0 w_0} -\frac{2\pi^2}{3w_0^2}.
\end{equation*}
Taking derivatives of $q^2(z)=iz/(1-e^{2\pi i z})$ and evaluating at $z=z_0$ shows that
\begin{equation*}
    \frac{q_1}{q_0}  = -\pi i +\frac{1}{2z_0}+\frac{\pi i}{w_0}, \qquad
    \frac{q_2}{q_0}  = -\frac{\pi^2}{2}  -\frac{\pi i}{2z_0}+\frac{2\pi^2}{w_0} -\frac{1}{8z_0^2}+\frac{ \pi i}{2z_0 w_0} -\frac{3\pi^2}{2w_0^2}.
\end{equation*}
Then by \eqref{r00},
\begin{equation*}
    \frac{f_1}{f_0}  = \frac{q_1}{q_0}-2\pi i\lambda, \qquad
    \frac{f_2}{f_0}  = \frac{q_2}{q_0} -2\pi i\lambda \frac{q_1}{q_0} -2\pi^2 \lambda^2.
\end{equation*}
Putting this all together with the quantities $a_0(f_{\mathcal A,\lambda}),$  $p_0^{1/2},$ $p_0$ and $f_0$ evaluated in \eqref{oad} -- \eqref{a0q}, and simplifying with \eqref{w0x}, finishes the proof.
\end{proof}

Table \ref{a1n1} gives examples of the accuracy of Theorem \ref{thma}, showing the approximation to $\mathcal A_1(N, -\lambda N)$ on the right side of \eqref{thmaa} for $N=1200$ and different values of $\lambda$ and $m$. The last column computes $\mathcal A_1(N, -\lambda N)$ directly from \eqref{cj1}.
\begin{table}[ht]
{\footnotesize
\begin{center}
\begin{tabular}{cc|cccc|c}
$N$ & $\lambda$ & $m=1$ & $m=2$ & $m=3$ &  $m=5$ & $\mathcal A_1(N, -\lambda N)$  \\ \hline
$1200$ & $1/3$ & $-1.60733\times 10^{30}$ &  $-1.60827 \times 10^{30}$ & $-1.60783 \times 10^{30}$ &   $-1.60784 \times 10^{30}$ &  $-1.60784 \times 10^{30}$\\
$1200$ & $1$ & $\phantom{-}1.89943\times 10^{30}$ &  $\phantom{-}1.71839 \times 10^{30}$ & $\phantom{-}1.72504 \times 10^{30}$ &   $\phantom{-}1.72506 \times 10^{30}$ &  $\phantom{-}1.72507 \times 10^{30}$ \\
$1200$ & $2$ & $-1.99478\times 10^{31}$ &  $-1.95514 \times 10^{31}$ & $-1.94125 \times 10^{31}$ &   $-1.94292 \times 10^{31}$ &  $-1.94291 \times 10^{31}$
\end{tabular}
\caption{The approximations of Theorem \ref{thma}  to $\mathcal A_1(N,-\lambda N)$.} \label{a1n1}
\end{center}}
\end{table}

\section{Proof of Theorem \ref{thmb}} \label{klmx2}
We prove Theorem \ref{thmb} in this section; it will follow directly from Propositions \ref{by},  \ref{cy},  \ref{dy} and \ref{ey}.
For $N \gqs 300$,  the indexing set $\farey_N-(\farey_{100} \cup \mathcal A(N))$ in \eqref{indx} may be partitioned into four pieces:
\begin{equation*}
     \mathcal B(101,N) \cup \mathcal C(N) \cup \mathcal D(N) \cup \mathcal E(N)
\end{equation*}
where
\begin{align}
    \mathcal C(N) & := \Bigl\{ h/k \ : \ N/2  <k \lqs N,  \ k \text{ odd}, \ h=2 \text{ \ or \ } h=k-2 \Bigr\}, \label{cnsub}\\
    \mathcal D(N) & := \Bigl\{ h/k \ : \ N/2  <k \lqs N,  \ k \text{ odd}, \ h=(k-1)/2 \text{ \ or \ } h=(k+1)/2 \Bigr\}, \label{dnsub}\\
    \mathcal E(N) & := \Bigl\{ h/k \ : \ N/3  <k \lqs N/2,  \ h=1 \text{ \ or \ } h=k-1 \Bigr\} \label{ensub}
\end{align}
and $\mathcal B(101,N)$ is what is left. We see next that the sum of $Q_{hk(-\lambda N)}(N)$ for $h/k \in \mathcal B(101,N)$ is small enough that we may bound the absolute value of each term. Doing this with the sums over $\mathcal C(N)$, $\mathcal D(N)$ and $\mathcal E(N)$ produces bounds that are larger than the main term of Theorem \ref{thma} and so we must use saddle-point methods for these.

\subsection{Bounds for $h/k \in \mathcal B(101,N)$}
\begin{prop} \label{by}
Let $\lambda^+$ be a positive real number. Suppose $N \in \Z_{\gqs 101}$ and  $\lambda N \in \Z$ for $\lambda$ satisfying $|\lambda| \lqs \lambda^+$. For an implied constant depending only on  $\lambda^+$,
\begin{equation}\label{shwb}
    \sum_{h/k \in \mathcal B(101,N)} Q_{hk(-\lambda N)}(N) =O(e^{0.055N}).
\end{equation}
\end{prop}
\begin{proof}
For $101 \lqs k \lqs N$ and $s := \lfloor N/k \rfloor$ we have from \cite[Prop. 3.4]{OS2} that
\begin{equation} \label{i7}
    |Q_{hk\sigma}(N)|  \lqs \frac{9}{k^3} \exp\left( N \frac{2  +  \log \left(\xi/2+ \xi' k/8 \right)}{k} +\frac{|\sigma|}{N} \right)\left|\spr{h/k}{N-s k} \right|
\end{equation}
for $\xi \approx 1.00038$ and $\xi' \approx 1.01041$. Then \eqref{i7} is used to prove, for an implied constant depending only on $\sigma$, that  $\sum_{h/k \in \mathcal B(101,N)}  Q_{hk\sigma}(N) = O(e^{0.055N})$. This is \cite[Thm. 3.5]{OS2}. When $\sigma = -\lambda N$, the only $\lambda$ dependence in \eqref{i7} is a factor $\exp(|\lambda|)$.
\end{proof}

\subsection{Bounds for $h/k \in \mathcal C(N)$}
This section may be read alongside Sections 5 and 6 of \cite{OS2}.
Set
\begin{equation}\label{grig}
\mathcal C_1(N,\sigma)  := \sum_{h/k \in \mathcal C(N)} Q_{hk\sigma}(N) = 2 \Re \sum_{\frac{N}{2}  <k \lqs N, \ k \text{ odd}}  Q_{2k\sigma}(N)
\end{equation}
where the equality in \eqref{grig} uses \eqref{defqnx} and requires $\sigma \in \Z$. We next wish to show
\begin{equation}\label{shw}
    \mathcal C_1(N,-\lambda N)=O(e^{WN})
\end{equation}
for a $W<U \approx 0.068$.
 In fact, the asymptotic expansion of $\mathcal C_1(N,\sigma)$ is found in \cite[Thm. 1.5]{OS2} and implies
\eqref{shw}, but with an implied constant depending on $\sigma$ and hence on $N$ when $\sigma$ takes the form $-\lambda N$. It is straightforward to rework the proof slightly (as we did for $\mathcal A_1(N,-\lambda N)$ in Section \ref{klmx}) and prove \eqref{shw} with an implied constant depending only on $\lambda$.
We give the details of this next.

The sum \eqref{grig} corresponds to $2N/k \in [2,4)$ and our treatment requires breaking this into two parts:
 $\mathcal C_2(N,\sigma)$ for $2N/k \in [2,3)$ and $\mathcal C^*_2(N,\sigma)$  with $2N/k \in [3,4)$.

We establish an intermediate result for $\mathcal C_2(N,-\lambda N)$ first, as follows.
Recall $p(z)$ from  \eqref{p(z)} and $g_\ell(z)$ from \eqref{grz}. Define
\begin{align}
  q_\mathcal C(z) & := \left( \frac{z }{2\sin(\pi z)}\right)^{1/2} \exp(-\pi i z/2 ),  \label{qbw}\\
  f_{\mathcal C,\lambda}(z) & := q_\mathcal C(z) \exp(-2\pi i \lambda z ), \label{rcz}\\
  v_\mathcal C(z;N,\sigma)  & := \frac{2\pi i \sigma z}N +\sum_{\ell=1}^{L-1} \frac{g_{\ell}(z)}{N^{2\ell-1}} \qquad (L= \lfloor 0.006\pi e \cdot N/2\rfloor). \label{vbw}
\end{align}

\begin{prop} \label{pxc2}
Let $\lambda^+$ be a positive real number. Suppose $N \in \Z_{\gqs 1}$ and  $\lambda N \in \Z$ for $\lambda$ satisfying $|\lambda| \lqs \lambda^+$.
For an implied constant depending only on $\lambda^+$,
\begin{equation}\label{oso}
    \mathcal C_2(N,-\lambda N)=\frac{1}{2N^{3/2}}  \Re \int_{2.01}^{2.49}
     \exp \bigl(-N \cdot p(z) \bigr) f_{\mathcal C,\lambda}(z) \exp \bigl(v_\mathcal C(z;N,0)\bigr) \, dz + O\left(e^{0.05N/2}\right).
\end{equation}
\end{prop}
\begin{proof}
The identity
\begin{multline}\label{b1n}
\mathcal C_2(N,\sigma) = \Re \sum_{k \text{ odd}, \ 2N/k \in [2,3)} \frac{-2}{k^2}
\exp\left( N\left[ \frac{\pi i}{2}\left(-\frac{2N}{k}+5 -2\frac k{2N} \right)\right]\right) \\
\times \exp\left(\frac{-\pi i}{2}\frac{2N}{k}\right) \exp\left( \frac 1N\left[ 2 \pi i \sigma\frac{2N}{k}\right]\right)
 \spr{2/k}{N-k}.
\end{multline}
is \cite[eq. (5.3)]{OS2}.
 Set $\hat{z}=\hat{z}(N,k):=2N/k$ and define
\begin{equation} \label{c3nx}
    \mathcal C_3(N,\sigma)  := \frac{-2}{N^{1/2}} \Re \sum_{k  \text{ odd} : \  \hat{z} \in  (2.01 ,  2.49)}
    \frac{1}{k^2} \exp \bigl(-N (p(\hat{z})-2\pi i/\hat{z})\bigr) q_{\mathcal C}(\hat{z}) \exp \bigl(v_\mathcal C (\hat{z};N,\sigma)\bigr).
\end{equation}
The sine product $\spr{2/k}{N-k}$ in \eqref{b1n} may be estimated precisely using Euler-Maclaurin summation as in \cite[Thm. 5.1]{OS2}. The result is that $\mathcal C_2(N,\sigma)$ and $\mathcal C_3(N,\sigma)$
differ by at most $O(e^{0.05N/2})$ for an absolute implied constant. For $\sigma = -\lambda N$ we easily obtain
\begin{equation*}
    \mathcal C_3(N,-\lambda N)  = \frac{-2}{N^{1/2}} \Re \sum_{k  \text{ odd} : \  \hat{z} \in  (2.01 ,  2.49)}
    \frac{1}{k^2} \exp \bigl(-N (p(\hat{z})-2\pi i/\hat{z})\bigr) f_{\mathcal C,\lambda}(\hat{z}) \exp \bigl(v_\mathcal C (\hat{z};N,0)\bigr).
\end{equation*}
It is shown in \cite[Thm. 5.4]{OS2} that $q_{\mathcal C}(z)$ and $v_\mathcal C (z;N,0)$ are holomorphic and absolutely bounded on a domain containing the box
$\mathbb B_2$, defined in \eqref{box}.
Hence,
\begin{equation}\label{soot}
    f_{\mathcal C,\lambda}(z) \exp \bigl(v_\mathcal C (z;N,0)\bigr) \ll 1 \quad \text{for} \quad z \in \mathbb B_2
\end{equation}
with an implied constant depending only on $\lambda^+$. Using this bound, the proofs of Propositions 5.6 and 5.7 in \cite{OS2} go through. This gives the desired result, expressing $\mathcal C_2(N,-\lambda N)$ as the integral in \eqref{oso}.
\end{proof}

The  second component, $\mathcal C^*_2(N,\sigma)$, is treated in a similar way as follows.
Define
\begin{align}
  p_1(z) &:= \frac 1{2\pi i z} \Bigl[-\li(e^{2\pi i z}) +\li(1) +4\pi^2 \Bigr], \notag\\
  q_\mathcal C^*(z) &:= e^{-3\pi i/4} \sqrt{z}, \notag\\
  f^*_{\mathcal C, \lambda}(z) & := q_\mathcal C^*(z) \exp(-2\pi i \lambda z ), \label{fcsl}\\
  v_\mathcal C^*(z;N,\sigma) &:= \frac{\pi i(16\sigma+1) z}{8(N+1/2)} + \sum_{\ell=1}^{L-1} \frac{g_{\ell}(z)}{(2(N+1/2))^{2\ell-1}}-\sum_{\ell=1}^{L^*-1} \frac{g_{\ell}(z)}{(N+1/2)^{2\ell-1}} \notag
\end{align}
for $L=\lfloor 0.006\pi e \cdot 2N/3\rfloor$ and $L^*=\lfloor 0.006\pi e \cdot N/3\rfloor$.
\begin{prop}
Let $\lambda^+$ be a positive real number. Suppose $N \in \Z_{\gqs 1}$ and  $\lambda N \in \Z$ for $\lambda$ satisfying $|\lambda| \lqs \lambda^+$.
For an implied constant depending only on $\lambda^+$,
\begin{multline}\label{intc2s}
    \mathcal C^*_2(N,-\lambda N)=\frac{-1}{4(N+1/2)^{3/2}}   \Im \int_{3.01}^{3.49}  \exp\bigl(-(N+1/2) p_1(z) \bigr) f^*_{\mathcal C, \lambda}(z) \exp \bigl(v_\mathcal C^*(z;N,\lambda/2)\bigr) \, dz \\
    + O(e^{0.05N/3}).
\end{multline}
\end{prop}
\begin{proof}
This time we  set $\hat{z}:=2(N+1/2)/k$ and define
\begin{multline} \label{c3nxs}
    \mathcal C_3^*(N,\sigma)  := \frac{1}{(N+1/2)^{1/2}}
     \Re \sum_{k \text{ odd} \ : \ \hat{z} \in  (3.01, 3.49)}
    \frac{(-1)^{(k+1)/2}}{k^2} \\
    \times\exp \bigl(-(N+1/2) (p_1(\hat{z})-3\pi i/\hat{z})\bigr)
    q^*_{\mathcal C}(\hat{z}) \exp \bigl(v_\mathcal C^*(\hat{z};N,\sigma)\bigr).
\end{multline}
It is shown in \cite[eq. (6.14)]{OS2} that $\mathcal C^*_2(N,\sigma)$ and $\mathcal C^*_3(N,\sigma)$
differ by at most $O(e^{0.05N/3})$ for an absolute implied constant. We may replace $q^*_{\mathcal C}(z) \exp \bigl(v_\mathcal C^*(z;N,\sigma)\bigr)$ in \eqref{c3nxs} by $f^*_{\mathcal C, \lambda}(z) \exp \bigl(v_\mathcal C^*(z;N,\lambda/2)\bigr)$ since  $\sigma = -\lambda N$.
Put
\begin{equation} \label{gazi}
    g_{\mathcal C, \ell}(z):=g_\ell(z)(2^{-(2\ell-1)}-1)
\end{equation}
and the next result is \cite[Corollary 6.4]{OS2}.
\begin{lemma} \label{acdcy}
For $\delta=0.007$ and all $z\in \C$ with $3+\delta \lqs \Re(z) \lqs 3.5-\delta$, we have
\begin{equation*}
    v_\mathcal C^*(z;N,\sigma) = \frac{\pi i(16\sigma+1) z}{8(N+1/2)} +\sum_{\ell=1}^{d-1} \frac{g_{\mathcal C, \ell}(z)}{(N+1/2)^{2\ell-1}} + O\left( \frac 1{N^{2d-1}}\right)
\end{equation*}
 where $2 \lqs d\lqs L^*=\lfloor 0.006 \pi e \cdot N/3 \rfloor$ and the implied constant depends only on  $d$.
\end{lemma}

With the above lemma we may show that $f^*_{\mathcal C, \lambda}(z) \exp \bigl(v_\mathcal C^*(z;N,\lambda/2)\bigr)$ is holomorphic on a domain containing $\mathbb B_3$,
 and that
\begin{equation}\label{mrw}
    f^*_{\mathcal C, \lambda}(z) \exp \bigl(v_\mathcal C^*(z;N,\lambda/2)\bigr) \ll 1 \quad \text{for} \quad z \in \mathbb B_3
\end{equation}
with an implied constant depending only on $\lambda^+$.
The proof in Sections 6.2 and 6.3 of \cite{OS2} now goes through and we obtain our integral representation \eqref{intc2s}.
\end{proof}

\begin{prop} \label{cy}
Let $\lambda^+$ be a positive real number. Suppose $N \in \Z_{\gqs 1}$ and  $\lambda N \in \Z$ for $\lambda$ satisfying $|\lambda| \lqs \lambda^+$.
For an implied constant depending only on $\lambda^+$,
\begin{equation}\label{shw2}
    \mathcal C_1(N,-\lambda N)=O(e^{0.036N}).
\end{equation}
\end{prop}
\begin{proof}
We bound the integral representation \eqref{oso} of $\mathcal C_2(N,-\lambda N)$ by first moving the path of integration to one going through a saddle-point of $p(z)$. By Theorem \ref{disol}, the unique solution to $p'(z)=0$ for $1.5<\Re(z)<2.5$ is
\begin{equation*}
    z_1:=2+\frac{\log \bigl(1-w(0,-2)\bigr)}{2\pi i} \approx 2.20541 + 0.345648 i.
\end{equation*}
Let  $c=1+i \Im(z_1)/\Re(z_1)$ and make $\mathcal Q$ the polygonal path between the points $2.01$, $2.01 c$, $2.49 c$ and $2.49$, as shown in Figure \ref{pthpqs}, passing through $z_1$. It is proved in \cite[Thm. 5.8]{OS2} that
\begin{equation}\label{aqwc}
    \Re(p(z)-p(z_1))>0 \quad \text{for all} \quad z \in \mathcal Q, \ z \neq z_1.
\end{equation}
Therefore, recalling \eqref{soot},
\begin{equation}\label{cnn2}
\mathcal C_2(N,-\lambda N) \ll e^{U_\mathcal C N} \quad \text{for} \quad U_\mathcal C:=-\Re(p(z_1))=-\log|w(0,-2)| \approx 0.0256706.
\end{equation}

Bounding the integral representation \eqref{intc2s} of $\mathcal C^*_2(N,-\lambda N)$ is achieved similarly. By Theorem \ref{disol}, the unique solution to $p'_1(z)=0$ for $2.5<\Re(z)<3.5$ is
\begin{equation*}
    z_3:=3+\frac{\log \bigl(1-w(1,-3)\bigr)}{2\pi i} \approx 3.08382 - 0.0833451  i.
\end{equation*}
Let  $c=1+i \Im(z_3)/\Re(z_3)$ and make $\mathcal S$ the polygonal path between the points $3.01$, $3.01 c$, $3.49 c$ and $3.49$, as shown in Figure \ref{pthpqs}, passing through $z_3$. As seen  in \cite[Sect. 6.3]{OS2},
\begin{equation*}%\label{aqwc}
    \Re(p_1(z)-p_1(z_3))>0 \quad \text{for all} \quad z \in \mathcal S, \ z \neq z_3.
\end{equation*}
Therefore, recalling \eqref{mrw},
\begin{equation}\label{cnn2s}
\mathcal C^*_2(N,-\lambda N) \ll e^{U^*_\mathcal C N} \quad \text{for} \quad U^*_\mathcal C:=-\Re(p(z_3))= -\log|w(1,-3)| \approx 0.0356795.
\end{equation}

Together, \eqref{cnn2} and \eqref{cnn2s} prove \eqref{shw2}.
\end{proof}

\subsection{Bounds for $h/k \in \mathcal D(N)$}
This section may be read alongside Section 7 of \cite{OS2}.
Write
\begin{equation}\label{suma}
\mathcal D_1(N,\sigma)  := \sum_{h/k \in \mathcal D(N)} Q_{hk\sigma}(N) = 2 \Re \sum_{\frac{N}{2}  <k \lqs N, \ k \text{ odd}}  Q_{(\frac{k-1}{2})k\sigma}(N).
\end{equation}
Recall $g_\ell(z)$ from \eqref{grz} and define
\begin{align*}
    q_\mathcal D(z) & :=  \left( \frac{z }{2\sin(\pi(z -1)/2)}\right)^{1/2} \exp\left( -\frac{\pi i}{4}\left(z+3\right)\right),  \\
     q^*_\mathcal D(z) & :=  2\sin(\pi(z -1)/2) \left( \frac{z }{2\sin(\pi(z -1)/2)}\right)^{1/2} \exp\left( \frac{\pi i}{4}\left(z-1\right)\right), \\
    g_\ell^*(z) & :=  -\frac{B_{2\ell}}{(2\ell)!} \left( \pi z/2\right)^{2\ell-1} \cot^{(2\ell-2)}\left(\pi (z-1)/2\right),\\
    v_\mathcal D(z;N, \sigma) & :=  \frac{\pi i \sigma z}{N} + \sum_{\ell=1}^{L-1} \frac{g_{\ell}(z)-g^*_{\ell}(z)}{N^{2\ell-1}}
    +\sum_{\ell=1}^{L^*-1} \frac{2 g^*_{\ell}(z) - g_{\ell}(z)}{(N/2)^{2\ell-1}}
\end{align*}
for $L:=\lfloor 0.006\pi e \cdot N\rfloor$ and $L^*:=\lfloor 0.006\pi e \cdot N/2\rfloor$. Also set
\begin{equation}\label{fdls}
    f_{\mathcal D, \lambda}(z)  := q_\mathcal D(z) \exp(-\pi i \lambda z),
    \qquad f^*_{\mathcal D, \lambda}(z)  := i q^*_\mathcal D(z) \exp(-\pi i \lambda z).
\end{equation}

\begin{prop} %\label{dy}
Let $\lambda^+$ be a positive real number. Suppose $N \in \Z_{\gqs 1}$ and  $\lambda N \in \Z$ for $\lambda$ satisfying $|\lambda| \lqs \lambda^+$. Then
\begin{equation}\label{shw3o}
    \mathcal D_1(N,-\lambda N)=
    -\Re \int_{1.01}^{1.49}
    \frac{\exp \bigl(-N \cdot p(z)/2 \bigr)}{ N^{3/2}}  f_{\mathcal D, \lambda}(z) \exp \bigl(v_\mathcal D(z;N,0)\bigr) \, dz + O(e^{0.05N/2})
\end{equation}
for $N$ odd. Also
\begin{equation}\label{shw3e}
    \mathcal D_1(N,-\lambda N)= -\Re \int_{1.01}^{1.49}
    \frac{\exp \bigl(-(N+1) \cdot p(z)/2 \bigr)}{ (N+1)^{3/2}}  f^*_{\mathcal D, \lambda}(z) \exp \bigl(v_\mathcal D(z;N+1,\lambda)\bigr) \, dz + O(e^{0.05N/2})
\end{equation}
for $N$ even. The implied constants in \eqref{shw3o} and \eqref{shw3e} depend only on $\lambda^+$,
\end{prop}
\begin{proof}
For $N$ odd, the summands in \eqref{suma} satisfy the identity
\begin{multline}\label{nodd}
    Q_{\left(\frac{k-1}{2}\right) k\sigma}(N)= \frac{1}{k^2}
\exp\left( N\left[ \frac{\pi i}{4}\left(\frac{N}{k}+1+\frac{2k}N \right)\right]\right)
\\
\times
\exp\left( \frac{\pi i}{4}\left(\frac{N}{k}+3\right)\right)
\exp\left( \frac 1N\left[ -\pi i \sigma \frac{N}{k}\right]\right)
\spr{(k-1)/2k}{N-k}
\end{multline}
which is \cite[Eq. (7.2)]{OS2}. Next put $\hat{z}:=N/k$ and
\begin{equation} \label{dn2x}
    \mathcal D_2(N,\sigma)  :=  \frac {-2}{N^{1/2}} \Re \sum_{k  \text{ odd}\ : \ \hat{z} \in  \left(1.01, \ 1.49\right)}
     \frac{1}{k^2} \exp \bigl(-N \cdot p(\hat{z})/2 \bigr) q_{\mathcal D}(\hat{z})  \exp \bigl(v_\mathcal D(\hat{z};N,\sigma)\bigr)
\end{equation}
for $N$ odd. It is shown in \cite[Eq. (7.30)]{OS2} that $\mathcal D_1(N,\sigma)$ and $\mathcal D_2(N,\sigma)$ differ by at most $O(e^{0.05N/2})$ for an absolute implied constant. With $\sigma = -\lambda N$ we substitute
\begin{equation*}
     f_{\mathcal D, \lambda}(z)  \exp \bigl(v_\mathcal D(z;N,0)\bigr) = q_{\mathcal D}(z)  \exp \bigl(v_\mathcal D(z;N,\sigma).
\end{equation*}
in \eqref{dn2x}. It follows from \cite[Thm. 7.5]{OS2} that $f_{\mathcal D, \lambda}(z)  \exp \bigl(v_\mathcal D(z;N,0)\bigr)$ is holomorphic on a domain containing
$\mathbb B_1$ and
\begin{equation}\label{mrwd}
    f_{\mathcal D, \lambda}(z)  \exp \bigl(v_\mathcal D(z;N,0)\bigr) \ll 1 \quad \text{for} \quad z \in \mathbb B_1
\end{equation}
with an implied constant depending only on $\lambda^+$.
The proof in Section 7.2  of \cite{OS2} now goes through and we obtain our integral representation \eqref{shw3o}.
The case where $N$ is even is very similar - see Section 7.3  of \cite{OS2} - using
\begin{multline} \label{dn2xe}
    \mathcal D_2(N,\sigma)  :=  \frac {-2}{(N+1)^{1/2}} \Re \sum_{k  \text{ odd}\ : \ \hat{z} \in  \left(1.01, \ 1.49\right)}
     \frac{(-1)^{(k+1)/2}}{k^2} \\
     \times\exp \bigl(-(N+1) \left[\frac{p(\hat{z})}{2} +\frac{\pi i}{2\hat{z}}\right]\bigr) q^*_{\mathcal D}(\hat{z})  \exp \bigl(v_\mathcal D(\hat{z};N+1,\sigma)\bigr) \qquad (\text{$N$ even}).
\end{multline}
\end{proof}

\begin{prop} \label{dy}
Let $\lambda^+$ be a positive real number. Suppose $N \in \Z_{\gqs 1}$ and  $\lambda N \in \Z$ for $\lambda$ satisfying $|\lambda| \lqs \lambda^+$.
Then for an implied constant depending only on $\lambda^+$,
\begin{equation}\label{shw3}
    \mathcal D_1(N,-\lambda N)=O(e^{0.035N}).
\end{equation}
\end{prop}
\begin{proof}
Change the path of integration in \eqref{shw3o} and \eqref{shw3e} to $\mathcal P$, as described in Section \ref{klmx}. We saw there that
\begin{equation*}%\label{aqw}
    \Re(p(z)-p(z_0))>0 \quad \text{for all} \quad z \in \mathcal P, \ z \neq z_0.
\end{equation*}
 It now follows from \eqref{mrwd} (and the analogous bound for $N$ even) that
\begin{equation}\label{dnn}
\mathcal D_1(N,-\lambda N) \ll e^{U N/2} \quad \text{for} \quad U :=-\Re(p(z_0))=-\log|w(0,-1)| \approx 0.0680762
\end{equation}
(as in \eqref{abuv}, \eqref{adzx}) and we have completed the proof.
\end{proof}

\subsection{Bounds for $h/k \in \mathcal E(N)$}
This section may be read alongside Section 8 of \cite{OS2}.
Put
\begin{equation} \label{e1su}
    \mathcal E_1(N,\sigma) := \sum_{h/k \in \mathcal E(N)} Q_{hk\sigma}(N) = 2 \Re \sum_{ \frac{N}{3}  <k \lqs \frac{N}{2}}  Q_{1k\sigma}(N).
\end{equation}
To express \eqref{e1su} as an integral we need the next definitions.
First set
\begin{equation*}
    \tilde{g}_\ell(z)  := \frac{B_{2\ell }}{(2\ell )!} \left( \pi z\right)^{2\ell -1}\left\{ \pi z \cot^{(2\ell -1)}\left(\pi z\right) +
    (2 \ell -1) \cot^{(2\ell -2)}\left(\pi z\right) \right\}.
\end{equation*}
Then we define the function $\phi_{\sigma, m}(z)$ in the following way:
\begin{align*}
    \phi_{\sigma, 0}(z) & :=\frac 1{4\pi^2} \left[\li(1)- \li(e^{2\pi i z}) +6\pi^2 -2\pi i z \log(1-e^{2\pi i z}) \right],\\
    \phi_{\sigma, 1}(z) & :=\frac{z^2 \cot(\pi z)}{4i} + \frac{z^2}{4} - \frac{5z}{4\pi i},\\
    \phi_{\sigma, 2}(z) & :=-\sigma z^2 + \frac{z \tilde{g}_1(z)}{2\pi i}
\end{align*}
and for $m \in \Z_{\gqs 3}$
\begin{equation*}
    \phi_{\sigma, m}(z)  := \begin{cases} z \tilde{g}_{m/2}(z)/(2\pi i) & \text{ if $m$ is even}\\
    0 & \text{ if $m$ is odd}.
    \end{cases}
\end{equation*}
Clearly, only $\phi_{\sigma, 2}(z)$ depends on $\sigma$. With $q_{\mathcal C}(z)$ given in \eqref{qbw}, put
\begin{equation}
    q_{\mathcal E}(z;N, \sigma)  :=   q_{\mathcal C}(z) \times  \sum_{\ell=0}^{2L-1} \frac{\phi_{\sigma, \ell}(z)}{N^{\ell}}  \label{qdx}
\end{equation}
for $L := \lfloor 0.006\pi e \cdot N/2\rfloor$ and finally define
\begin{equation}
    f_{\mathcal E,\lambda}(z;N)  := q_{\mathcal E}(z;N, -\lambda N) \exp(-2\pi i \lambda z ). \label{rez}
\end{equation}

\begin{prop} %\label{ey}
Let $\lambda^+$ be a positive real number. Suppose $N \in \Z_{\gqs 1}$ and  $\lambda N \in \Z$ for $\lambda$ satisfying $|\lambda| \lqs \lambda^+$.
Then for an implied constant depending only on $\lambda^+$,
\begin{equation}\label{shw4}
    \mathcal E_1(N,-\lambda N)=\frac{1}{N^{3/2}}  \Re \int_{2.01}^{2.49}
     \exp \bigl(-N \cdot p(z) \bigr) f_{\mathcal E,\lambda}(z;N) \exp \bigl(v_\mathcal C(z;N, 0)\bigr) \, dz + O\left(N e^{0.05N/2}\right).
\end{equation}
\end{prop}
\begin{proof}
For  $N/3<k\lqs N/2$ and $\sigma \in \Z$, we know by \cite[Eqs. (8.3), (8.4)]{OS2}
that
\begin{multline} \label{cyan}
     Q_{1k\sigma}(N) =\frac{1}{2k^2} \phi(N,k,\sigma)
\exp\left( N\left[ \frac{-i\pi}{2}\left(\frac{N}{k}-1 +2\frac k N \right)\right]\right)\\
\times
\exp\left( \frac{-i\pi}2 \frac{N}{k}\right)
\exp\left( \frac 1N\left[ 2 i\pi \sigma\frac{N}{k}\right]\right)
 \spr{1/k}{N-2k}
\end{multline}
where
\begin{equation} \label{pink}
    \phi(N,k,\sigma) := \frac{1}{4k^2}(N^2+N-4\sigma) + \frac{1}{2\pi ik}\sum_{1\lqs j\lqs N,\ k \nmid j} \frac{\pi j}{k} \cot \left(\frac{\pi j}{k}\right).
\end{equation}
Let $\hat{z}=N/k$ again.
It is shown in \cite[Sections 8.3, 8.4]{OS2} that
\begin{equation*}
    \phi(N,k,\sigma) = \sum_{\ell=0}^{2L-1} \frac{\phi_{\sigma, \ell}(\hat{z})}{N^{\ell}} + \frac{ \varepsilon_L(N-2k,1/k) }{2\pi i k}
\end{equation*}
with error term $\varepsilon_L$ satisfying
\begin{align}
 \spr{1/k}{N-2k} \frac{ \varepsilon_L(N-2k,1/k) }{2\pi i k}  & = O(N e^{0.05N/2}), \label{slmad}\\
    \frac{ \varepsilon_L(N-2k,1/k) }{2\pi i k}  & = O(N) \label{slmbd}
\end{align}
where $L = \lfloor 0.006\pi e \cdot N/2\rfloor$ as in the definition of $q_{\mathcal E}$.
Recalling \eqref{vbw}, we next set
\begin{equation} \label{sroe}
    \mathcal E_2(N, \sigma)  := \frac{1}{N^{1/2}}
    \Re \sum_{k \ : \  \hat{z} \in  (2.01 ,  2.49)}
    \frac{1}{k^2} \exp \bigl(-N (p(\hat{z})-2\pi i/\hat{z})\bigr)
    q_{\mathcal E}(\hat{z};N, \sigma) \exp \bigl(v_\mathcal C(\hat{z};N, \sigma)\bigr).
\end{equation}
It is proved in \cite[Prop. 8.5]{OS2} (using \eqref{slmad} and \eqref{slmbd}) that $\mathcal E_2(N, \sigma) = \mathcal E_1(N, \sigma)+ O(N e^{0.05 N/2})$ for an implied constant depending only on $\sigma$. To adapt this to the case that $\sigma = -\lambda N$, we need the next result.

\begin{lemma} \label{phib}
For $N/3<k\lqs N/2$, $|\lambda| \lqs \lambda^+$ and an  implied constant depending only on $\lambda^+$
\begin{equation*}
    \phi(N,k,-\lambda N)=O(N).
\end{equation*}
\end{lemma}
\begin{proof}
Verify that $|\cot(\pi j/k)| < 2k/\pi$ when $k \nmid j$. This implies
\begin{equation*}
    \phi(N,k,-\lambda N) \ll 1+\frac{|\lambda| N}{k^2} + \frac{1}{k} \sum_{j=1}^N \frac jk k \ll   \frac{|\lambda|}{N}+N. \qedhere
\end{equation*}
\end{proof}

The  proof of  \cite[Prop. 8.5]{OS2}  with $\sigma = -\lambda N$ and the bound from Lemma \ref{phib} now shows that
\begin{equation*}
    \mathcal E_2(N, -\lambda N) = \mathcal E_1(N, -\lambda N)+ O(N e^{0.05 N/2})
\end{equation*}
for an implied constant depending only on $\lambda^+$. Replacing $\sigma$ by $-\lambda N$ also lets us rewrite $q_{\mathcal E}$ as
\begin{equation*}
    q_{\mathcal E}(z;N, -\lambda N)  =   q_{\mathcal C}(z) \times  \sum_{\ell=0}^{2L-1} \frac{\phi^*_{\lambda, \ell}(z)}{N^{\ell}} % \label{qdx}
\end{equation*}
with
\begin{equation}\label{phis}
    \phi^*_{\lambda, \ell}(z):= \begin{cases}
    \phi_{0, \ell}(z) & \text{ if }\quad \ell \in \Z_{\gqs 0}, \ell \neq 1\\
    \phi_{0, 1}(z)+\lambda z^2 &  \text{ if }\quad  \ell =1.
    \end{cases}
\end{equation}
In this way, $\phi^*_{\lambda, \ell}(z)$ is always independent of $N$, and only  depends on $\lambda$ when $\ell=1$.
We easily obtain
\begin{equation*}
    \mathcal E_2(N, -\lambda N)  = \frac{1}{N^{1/2}}
    \Re \sum_{k \ : \  \hat{z} \in  (2.01 ,  2.49)}
    \frac{1}{k^2} \exp \bigl(-N (p(\hat{z})-2\pi i/\hat{z})\bigr)
    f_{\mathcal E,\lambda}(\hat{z};N) \exp \bigl(v_\mathcal C(\hat{z};N, 0)\bigr).
\end{equation*}
With the help of \cite[Lemma 8.6]{OS2} we see that $f_{\mathcal E,\lambda}(z;N) \exp \bigl(v_\mathcal C(z;N, 0)\bigr)$ is holomorphic on a domain containing
$\mathbb B_2$ and that
\begin{equation}\label{mrwe}
    f_{\mathcal E,\lambda}(z;N) \exp \bigl(v_\mathcal C(z;N, 0)\bigr) \ll 1 \quad \text{for} \quad z \in \mathbb B_2
\end{equation}
with an implied constant depending only on $\lambda^+$.
The proof in Section 8.4  of \cite{OS2} now goes through and we obtain our integral representation \eqref{shw4}.
\end{proof}

\begin{prop} \label{ey}
Let $\lambda^+$ be a positive real number. Suppose $N \in \Z_{\gqs 1}$ and  $\lambda N \in \Z$ for $\lambda$ satisfying $|\lambda| \lqs \lambda^+$.
Then for an implied constant depending only on $\lambda^+$,
\begin{equation}\label{shw4ex}
    \mathcal E_1(N,-\lambda N)=O(e^{0.026N}).
\end{equation}
\end{prop}
\begin{proof} As in the first part of the proof of Proposition \ref{cy}, we move the path of integration of \eqref{shw4} to $\mathcal Q$ passing through
the saddle-point   $z_1$. As we have seen in \eqref{aqwc} and \eqref{cnn2}, $$-\Re(p(z)) \lqs -\Re(p(z_1)) = -\log|w(0,-2)| \approx 0.0256706$$ for $z \in \mathcal Q$. Combining this bound with \eqref{mrwe} completes the proof.
\end{proof}

This last proposition completes the proof of Theorem \ref{thmb}. This finishes the proof of Theorem \ref{ma2} as well.

%We will see that the largest contributions to the right side of \eqref{slee} correspond to $h/k$ in $\mathcal A(N)$ and $\farey_{100}$.

\section{Asymptotic expansions for $\mathcal C_1$, $\mathcal D_1$ and $\mathcal E_1$} \label{sec7}
We may continue the analysis of $\mathcal C_1(N,-\lambda N)$, $\mathcal D_1(N,-\lambda N)$ and $\mathcal E_1(N,-\lambda N)$ to obtain their asymptotic expansions. While not necessary to prove Theorem \ref{thmb}, this allows us to numerically verify  our work, as we see in Tables \ref{c2n1} - \ref{e1n1}, and also suggests some further results that we describe in Section \ref{fur}.

Recall the dilogarithm zeros $w_0=w(0,-1),$ $w(0,-2)$ and $w(1,-3)$ from Section \ref{dise}. Starting with Proposition \ref{pxc2}, a similar proof to Theorem \ref{thma} shows the asymptotic expansion of $\mathcal C_2(N,-\lambda N)$, the first component of $\mathcal C_1(N,-\lambda N)$. Details are much the same as  \cite[Sect. 5.4]{OS2}. The functions $u_{0,j}$ and $f_{\mathcal C,\lambda}$ are defined in \eqref{uiz} and \eqref{rcz}, respectively.

\begin{theorem} \label{c2se} Let $\lambda^+$ be a positive real number. Suppose $N \in \Z_{\gqs 1}$ and  $\lambda N \in \Z$ for $\lambda$ satisfying $|\lambda| \lqs \lambda^+$.
We have
\begin{equation} \label{presc}
    \mathcal C_2(N,-\lambda N) = \Re\left[\frac{w(0,-2)^{-N}}{N^{2}} \left( c_{0}(\lambda)+\frac{c_{1}(\lambda)}{N}+ \dots +\frac{c_{m-1}(\lambda)}{N^{m-1}}\right)\right] + O\left(\frac{|w(0,-2)|^{-N}}{N^{m+2}}\right)
\end{equation}
for an implied constant depending only on  $\lambda^+$ and $m$. The functions $c_{t}(\lambda)$ are given as follows, where $\alpha$ depends on $p$ with saddle-point $z_1$,
\begin{equation} \label{cso}
    c_{t}(\lambda) = \sum_{s=0}^t \G\left(s+\frac 12\right) \alpha_{2s}(f_{\mathcal C,\lambda} \cdot u_{0,t-s}).
\end{equation}
\end{theorem}

\begin{table}[ht]
{\footnotesize
\begin{center}
\begin{tabular}{cc|cccc|c}
$N$ & $\lambda$ & $m=1$ & $m=2$ & $m=3$ &  $m=5$ & $\mathcal C_2(N, -\lambda N)$  \\ \hline
$1200$ & $1/3$ & $\phantom{-}3.82015\times 10^7$ &  $\phantom{-}3.94694 \times 10^7$ & $\phantom{-}3.92935 \times 10^7$ &   $\phantom{-}3.93016 \times 10^7$ &  $\phantom{-}3.93016 \times 10^7$ \\
$1200$ & $1$ & $-7.55408\times 10^7$ &  $-9.44498 \times 10^7$ & $-8.92043 \times 10^7$ &  $-8.98686 \times 10^7$ &  $-8.98650 \times 10^7$ \\
$1200$ & $2$ & $-4.19152\times 10^9$ &  $-3.06226 \times 10^9$ & $-3.20956 \times 10^9$ &  $-3.20879 \times 10^9$ &  $-3.21242 \times 10^9$
\end{tabular}
\caption{The approximations of Theorem \ref{c2se} to $\mathcal C_2(N,-\lambda N)$.} \label{c2n1}
\end{center}}
\end{table}

A comparison of both sides of \eqref{presc} in Theorem \ref{c2se} with  some different values of $N,$  $\lambda$ and  $m$ is shown in Table \ref{c2n1}.
The asymptotic expansion of $\mathcal C^*_2(N,-\lambda N)$, the second component of $\mathcal C_1(N,-\lambda N)$, is given next based on a similar development in \cite[Sect. 6.4]{OS2}. Recall $f^*_{\mathcal C,\lambda}$ from \eqref{fcsl} and $g_{\mathcal C,\ell}$ from \eqref{gazi}. Set $u^*_{\sigma,  0}:=1$ and for $j \in \Z_{\gqs 1}$ put
\begin{equation*}
    u^*_{\sigma, j}(z):=\sum_{m_1+3m_2+5m_3+ \dots =j}\frac{(\pi i (16\sigma+1)z/8 +g_{\mathcal C, 1}(z))^{m_1}}{m_1!}\frac{g_{\mathcal C,  2}(z)^{m_2}}{m_2!} \cdots \frac{g_{\mathcal C, j}(z)^{m_j}}{m_j!}.
\end{equation*}

\begin{theorem} \label{c2ses} Let $\lambda^+$ be a positive real number. Suppose $N \in \Z_{\gqs 1}$ and  $\lambda N \in \Z$ for $\lambda$ satisfying $|\lambda| \lqs \lambda^+$.
We have
\begin{equation} \label{prescs}
    \mathcal C^*_2(N,-\lambda N) = \Re\left[\frac{w(1,-3)^{-N}}{N^{2}} \left( c^*_{0}(\lambda)+\frac{c^*_{1}(\lambda)}{N}+ \dots +\frac{c^*_{m-1}(\lambda)}{N^{m-1}}\right)\right] + O\left(\frac{|w(1,-3)|^{-N}}{N^{m+2}}\right)
\end{equation}
for an implied constant depending only on  $\lambda^+$ and $m$. The functions $c^*_{t}(\lambda)$ are given as follows, where $\alpha$ depends on $p_1$ with saddle-point $z_3$,
\begin{equation*}
    c^*_{t}(\lambda) = \frac{w(1,-3)^{-1/2}}{2} \sum_{j=0}^t (-2)^{j-t}\binom{t+1}{j+1} \sum_{s=0}^j \G\left(s+\frac 12\right) \alpha_{2s}(i f^*_{\mathcal C,\lambda} \cdot u^*_{\lambda/2,j-s}).
\end{equation*}
\end{theorem}

\begin{table}[ht]
{\footnotesize
\begin{center}
\begin{tabular}{cc|cccc|c}
$N$ & $\lambda$ & $m=1$ & $m=2$ & $m=3$ &  $m=5$ & $\mathcal C^*_2(N, -\lambda N)$  \\ \hline
$1200$ & $1/3$ & $-3.41978 \times 10^{11}$ &  $-3.38619 \times 10^{11}$ & $-3.38622 \times 10^{11}$ &  $-3.38622 \times 10^{11}$ &  $-3.38622 \times 10^{11}$\\
$1200$ & $1$ & $-5.45974 \times 10^{11}$ &  $-5.36354 \times 10^{11}$ & $-5.36405 \times 10^{11}$ &  $-5.36404 \times 10^{11}$ &  $-5.36404 \times 10^{11}$ \\
$1200$ & $2$ & $-5.14590 \times 10^{11}$ &  $-5.00599 \times 10^{11}$ & $-5.00740 \times 10^{11}$ &  $-5.00737 \times 10^{11}$ &  $-5.00737 \times 10^{11}$
\end{tabular}
\caption{The approximations of Theorem \ref{c2ses} to $\mathcal C^*_2(N,-\lambda N)$.} \label{c2sn1}
\end{center}}
\end{table}

The detailed expansion of $\mathcal D_1(N,-\lambda N)$ is derived similarly to Section 7 of \cite{OS2}. We need $f_{\mathcal D,\lambda},$ $f^*_{\mathcal D,\lambda}$ from \eqref{fdls}. Also set
\begin{equation*} \label{gazid}
    g_{\mathcal D, \ell}(z):=g_\ell(z) - g_\ell^*(z)+2^{2\ell-1}\bigl(2g_\ell^*(z) - g_\ell(z)\bigr)
\end{equation*}
and
\begin{equation*}
    u_{\mathcal D, \sigma, j}(z):=\sum_{m_1+3m_2+5m_3+ \dots =j}\frac{(\pi i \sigma z +g_{\mathcal D, 1}(z))^{m_1}}{m_1!}\frac{g_{\mathcal D, 2}(z)^{m_2}}{m_2!} \cdots \frac{g_{\mathcal D, j}(z)^{m_j}}{m_j!},
\end{equation*}
with $u_{\mathcal D, \sigma, 0}=1$.

\begin{theorem} \label{d0c} Let $\lambda^+$ be a positive real number. Suppose $N \in \Z_{\gqs 1}$ and  $\lambda N \in \Z$ for $\lambda$ satisfying $|\lambda| \lqs \lambda^+$.
With $\overline{N}$ denoting $N \bmod 2$, we have
\begin{equation} \label{presd}
    \mathcal D_1(N,-\lambda N) = \Re\left[\frac{w_0^{-N/2}}{N^{2}} \left( d_{0}\bigl(\lambda, \overline{N}\bigr) +\frac{d_{1}\bigl(\lambda, \overline{N}\bigr)}{N}+ \dots +\frac{d_{m-1}\bigl(\lambda, \overline{N}\bigr)}{N^{m-1}}\right)\right] + O\left(\frac{|w_0|^{-N/2}}{N^{m+2}}\right)
\end{equation}
for an implied constant depending only on  $\lambda^+$ and $m$. The functions $d_{t}\bigl(\lambda, \overline{N}\bigr)$ are given as follows, where $\alpha$ depends on $p/2$ with saddle-point $z_0$,
\begin{align}
    d_{t}\bigl(\lambda, \overline{1}\bigr) & = -2 \sum_{s=0}^t  \G\left(s+\frac 12\right) \alpha_{2s}(f_{\mathcal D,\lambda} \cdot u_{\mathcal D, 0,t-s}), \label{dso}\\
    d_{t}\bigl(\lambda, \overline{0}\bigr) & = -2 w_0^{-1/2} \sum_{j=0}^t (-1)^{j-t}\binom{t+1}{j+1} \sum_{s=0}^j  \G\left(s+\frac 12\right) \alpha_{2s}(f^*_{\mathcal D,\lambda} \cdot u_{\mathcal D, \lambda,j-s}). \label{dso2}
\end{align}
\end{theorem}

\begin{table}[ht]
{\footnotesize
\begin{center}
\begin{tabular}{cc|cccc|c}
$N$ & $\lambda$ & $m=1$ & $m=2$ & $m=3$ &  $m=5$ & $\mathcal D_1(N, -\lambda N)$  \\ \hline
$1200$ & $1/3$ & $-1.54767 \times 10^{12}$ &  $-1.53845 \times 10^{12}$ & $-1.53755 \times 10^{12}$ &  $-1.53757 \times 10^{12}$ &  $-1.53757 \times 10^{12}$\\
$1200$ & $1$ & $\phantom{-}2.19568 \times 10^{12}$ &  $\phantom{-}2.20664 \times 10^{12}$ & $\phantom{-}2.20164 \times 10^{12}$ &  $\phantom{-}2.20181 \times 10^{12}$ &  $\phantom{-}2.20181 \times 10^{12}$ \\
$1200$ & $2$ & $-5.91241 \times 10^{12}$ &  $-5.62009 \times 10^{12}$ & $-5.60529 \times 10^{12}$ &  $-5.60869 \times 10^{12}$ &  $-5.60866 \times 10^{12}$\\ \hline
$1203$ & $1/3$ & $-2.17797 \times 10^{12}$ &  $-2.18005 \times 10^{12}$ & $-2.17903 \times 10^{12}$ &  $-2.17904 \times 10^{12}$ &  $-2.17904 \times 10^{12}$\\
$1203$ & $1$ & $\phantom{-}1.68218 \times 10^{12}$ &  $\phantom{-}1.74105 \times 10^{12}$ & $\phantom{-}1.73537 \times 10^{12}$ &  $\phantom{-}1.73545 \times 10^{12}$ &  $\phantom{-}1.73545 \times 10^{12}$ \\
$1203$ & $2$ & $-7.82854 \times 10^{12}$ &  $-7.68896 \times 10^{12}$ & $-7.65629 \times 10^{12}$ &  $-7.66011 \times 10^{12}$ &  $-7.66007 \times 10^{12}$
\end{tabular}
\caption{The approximations of Theorem \ref{d0c} to $\mathcal D_1(N,-\lambda N)$.} \label{d1n1}
\end{center}}
\end{table}

The asymptotic expansion of $\mathcal E_1(N,-\lambda N)$ is proved analogously to \cite[Sect. 8]{OS2}. The functions $f_{\mathcal E,\lambda},$ $\phi^*_{\lambda,k}$ and $u_{0,j}$ are defined in \eqref{rez}, \eqref{phis} and \eqref{uiz} respectively.

\begin{theorem} \label{eee} Let $\lambda^+$ be a positive real number. Suppose $N \in \Z_{\gqs 1}$ and  $\lambda N \in \Z$ for $\lambda$ satisfying $|\lambda| \lqs \lambda^+$.
We have
\begin{equation} \label{prese}
    \mathcal E_1(N,-\lambda N) = \Re\left[\frac{w(0,-2)^{-N}}{N^{2}} \left( e_{0}(\lambda)+\frac{e_{1}(\lambda)}{N}+ \dots +\frac{e_{m-1}(\lambda)}{N^{m-1}}\right)\right] + O\left(\frac{|w(0,-2)|^{-N}}{N^{m+2}}\right)
\end{equation}
for an implied constant depending only on  $\lambda^+$ and $m$. The functions $e_{t}(\lambda)$ are given as follows, where $\alpha$ depends on $p$ with saddle-point $z_1$,
\begin{equation} \label{eso}
    e_{t}(\lambda) = 2\sum_{s=0}^t \sum_{k=0}^{t-s}  \G\left(s+\frac 12\right)  \alpha_{2s}(f_{\mathcal E,\lambda} \cdot \phi^*_{\lambda,k} \cdot  u_{0,t-s-k}).
\end{equation}
\end{theorem}

\begin{table}[ht]
{\footnotesize
\begin{center}
\begin{tabular}{cc|cccc|c}
$N$ & $\lambda$ & $m=1$ & $m=2$ & $m=3$ &  $m=5$ & $\mathcal E_1(N, -\lambda N)$  \\ \hline
$1200$ & $1/3$ & $\phantom{-}1.14604\times 10^8$ &  $\phantom{-}1.18408 \times 10^8$ & $\phantom{-}1.17881 \times 10^8$ &   $\phantom{-}1.17905 \times 10^8$ &  $\phantom{-}1.17905 \times 10^8$\\
$1200$ & $1$ & $-2.26622\times 10^8$ &  $-2.83349 \times 10^8$ & $-2.67613 \times 10^8$ &  $-2.69606 \times 10^8$ &  $-2.69595 \times 10^8$ \\
$1200$ & $2$ & $-1.25746\times 10^{10}$ &  $-9.18677 \times 10^9$ & $-9.62868 \times 10^9$ &   $-9.62637 \times 10^9$ &  $-9.63726 \times 10^9$
\end{tabular}
\caption{The approximations of Theorem \ref{eee} to $\mathcal E_1(N,-\lambda N)$.} \label{e1n1}
\end{center}}
\end{table}

\section{When waves are good approximations to partitions}

\subsection{Proof of Theorem \ref{ma3}}

We first show upper and lower bounds for $p_N\bigl(\lambda N^2\bigr)$.

\begin{lemma}
For positive integers $N,$ $\lambda N^2$ and  absolute implied constants
\begin{equation}\label{pbun}
    \frac{1}{ \lambda N^2} e^{(2+\log \lambda)N} \ll p_N\bigl(\lambda N^2\bigr) \ll \frac{1}{ \lambda N^2} e^{\bigl(2\pi \sqrt{\lambda/6}\bigr)N}.
\end{equation}
\end{lemma}
\begin{proof} The upper bound follows from $p_N(n)\lqs p(n)$ and \eqref{pnexp}.
For the lower bound we start with
\begin{equation} \label{pbun2}
     \frac{1}{N!}\binom{N+n-1}{N-1} \lqs p_N(n)
\end{equation}
from, for example, \cite[p. 86]{Sze51}. Stirling's formula implies $\G(x) \sim \sqrt{2\pi}x^{x-1/2}e^{-x}$ as $x\to \infty$, so we see that the left side of \eqref{pbun2} is
\begin{equation} \label{pbun3}
     \frac{\G(N+n)}{\G(N)^2 \G(n) \cdot n\cdot N} \gg \frac{(\lambda e)^N}{\lambda N^2} \left(1+\frac{1}{\lambda N}\right)^{N(1+\lambda N)}
\end{equation}
for $n=\lambda N^2$. Now use that $(1+1/x)^{1+x} \gqs e$ for $x>0$ to show
$$
\left(1+\frac{1}{\lambda N}\right)^{N(1+\lambda N)} \gqs e^N
$$
and we have proved the lower bound.
\end{proof}

\begin{proof}[Proof of Theorem \ref{ma3}]
With $n=\lambda N^2$ in \eqref{slee} and \eqref{w100}, we have
\begin{equation} \label{strt}
   p_N\bigl(\lambda N^2\bigr) = \sum_{k=1}^{100} W_k\bigl(N,\lambda N^2\bigr) -  \sum_{h/k} Q_{hk(-\lambda N^2)}(N)
\end{equation}
where the second sum is over all $h/k \in \mathcal A(N) \cup \mathcal B(101,N) \cup \mathcal C(N) \cup \mathcal D(N) \cup \mathcal E(N)$. With
\eqref{a2ns} and \eqref{a1n}, the sum over $\mathcal A(N)$ is bounded by an absolute constant times the absolute value of
\begin{equation} \label{a2nsx}
     e^{0.05N}+\frac 2{N^{1/2}} \Im  \sum_{ k \ : \ \hat{z} \in  (1.01, 1.49)} \frac{(-1)^{k}}{k^2}  \exp\bigl(-N  (p(\hat{z})+2\pi i \lambda \hat{z}-\pi i/\hat{z}) \bigr) q\left(\hat{z} \right) \exp\bigl(v \left(\hat{z};N,0 \right) \bigr)
\end{equation}
where $\hat{z}=N/k$ as before. We know that $q\left(z \right) \exp\bigl(v \left(z;N,0 \right) \bigr) \ll 1$ by \eqref{soota} with $\lambda=0$. For $z=x\in \R$ and $\lambda \in \R$
\begin{equation}
    \Re\left[-p(z)-2\pi i \lambda z+\pi i/z\right]  = \Re\left[-p(z)\right] = \frac{\Im(\li(e^{2\pi i x}))}{2\pi x}
    = \frac{\cl(2\pi x)}{2\pi x} \label{bir}
\end{equation}
 and the right side of \eqref{bir} is bounded by $0.15$ for $1\lqs x \lqs 1.5$. Consequently, \eqref{a2nsx} is bounded by an absolute constant times
\begin{equation*}
     e^{0.05N}+ e^{0.15N}  \sum_{ k \ : \ \hat{z} \in  (1.01, 1.49)} \frac{1}{k^2}  \ll e^{0.15N}.
\end{equation*}
Similar reasoning, using \eqref{c3nx}, \eqref{c3nxs}, \eqref{dn2x}, \eqref{dn2xe} and \eqref{sroe}, shows the sums over $\mathcal C(N)$, $\mathcal D(N)$ and $\mathcal E(N)$ are also $O(e^{0.15N})$.

From the proof of Proposition \ref{by} we find that the sum over $\mathcal B(101,N)$ is $O(e^{(|\lambda|+0.055)N})$. Therefore
\begin{equation} \label{strt2}
   p_N\bigl(\lambda N^2\bigr) = \sum_{k=1}^{100} W_k\bigl(N,\lambda N^2\bigr) +  O\left(e^{0.15N} + e^{(|\lambda|+0.055)N}\right).
\end{equation}
Next, we divide both sides of \eqref{strt2} by $p_N\bigl(\lambda N^2\bigr)$ and use the lower bound of \eqref{pbun}.
Since
\begin{equation*}
    \max\{0.15, |\lambda|+0.055\}-2-\log \lambda < -0.109
\end{equation*}
for $0.2\lqs \lambda \lqs 2.9$ we obtain
\begin{equation*}
   p_N\bigl(\lambda N^2\bigr) \Big\backslash \left(\sum_{k=1}^{100} W_k\bigl(N,\lambda N^2\bigr)\right) = 1 +  O\left(e^{-0.1N}\right)
\end{equation*}
and the result follows easily.
\end{proof}

\subsection{Good approximations for $N$ or $n$ small} \label{good}

%*************************************
% Graphics max sine

%Parabolic

\SpecialCoor
\psset{griddots=5,subgriddiv=0,gridlabels=0pt}
\psset{xunit=0.02cm, yunit=0.7cm}
\psset{linewidth=1pt}
\psset{dotsize=2pt 0,dotstyle=*}

\begin{figure}[ht]
\begin{center}
\begin{pspicture}(180,-4)(700,4) %\psgrid

\savedata{\mydata}[
{{275., -4.17495}, {276.,
  3.65556}, {277., -3.94445}, {278.,
  3.31157}, {279., -3.69476}, {280., 3.0399}, {281., -3.43781}, {282.,
   2.83162}, {283., -3.15489}, {284., 2.6651}, {285., -2.858}, {286.,
  2.54514}, {287., -2.57105}, {288.,
  2.45273}, {289., -2.29163}, {290.,
  2.36096}, {291., -2.03649}, {292.,
  2.26913}, {293., -1.82523}, {294.,
  2.16286}, {295., -1.65322}, {296., 2.02844}, {297., -1.5248}, {298.,
   1.87441}, {299., -1.44346}, {300.,
  1.70288}, {301., -1.3927}, {302., 1.51602}, {303., -1.36337}, {304.,
   1.33229}, {305., -1.3475}, {306.,
  1.16155}, {307., -1.32617}, {308.,
  1.01038}, {309., -1.29007}, {310.,
  0.892163}, {311., -1.23569}, {312.,
  0.809991}, {313., -1.15528}, {314.,
  0.760299}, {315., -1.05155}, {316.,
  0.742712}, {317., -0.932961}, {318.,
  0.748003}, {319., -0.804822}, {320.,
  0.763291}, {321., -0.678392}, {322.,
  0.780179}, {323., -0.56632}, {324.,
  0.787526}, {325., -0.474868}, {326.,
  0.775252}, {327., -0.410724}, {328.,
  0.740955}, {329., -0.377726}, {330.,
  0.683325}, {331., -0.372143}, {332.,
  0.604298}, {333., -0.388921}, {334.,
  0.512222}, {335., -0.420911}, {336.,
  0.415702}, {337., -0.456641}, {338.,
  0.323522}, {339., -0.486401}, {340.,
  0.246263}, {341., -0.502375}, {342.,
  0.190836}, {343., -0.497348}, {344.,
  0.16065}, {345., -0.468928}, {346.,
  0.157177}, {347., -0.418945}, {348.,
  0.177013}, {349., -0.351197}, {350.,
  0.213153}, {351., -0.273325}, {352.,
  0.257758}, {353., -0.194889}, {354.,
  0.301119}, {355., -0.124516}, {356.,
  0.33377}, {357., -0.0704507}, {358.,
  0.34907}, {359., -0.0388126}, {360.,
  0.342571}, {361., -0.0315927}, {362.,
  0.313065}, {363., -0.0478248}, {364.,
  0.26371}, {365., -0.0833265}, {366.,
  0.200231}, {367., -0.130531}, {368.,
  0.130421}, {369., -0.18071}, {370.,
  0.0636406}, {371., -0.224982}, {372.,
  0.00847605}, {373., -0.255001}, {374., -0.028309}, {375.,
-0.265037}, {376., -0.0422126}, {377., -0.25258}, {378., -0.0324197},
{379., -0.218223}, {380., -0.00169828}, {381., -0.166198}, {382.,
  0.0444256}, {383., -0.103551}, {384.,
  0.0979497}, {385., -0.0387036}, {386., 0.149738}, {387.,
  0.0193225}, {388., 0.19123}, {389., 0.0624099}, {390.,
  0.215282}, {391., 0.084855}, {392., 0.217375}, {393.,
  0.0838612}, {394., 0.196539}, {395., 0.0600756}, {396.,
  0.155189}, {397., 0.0177324}, {398.,
  0.0988503}, {399., -0.0362617}, {400.,
  0.0355468}, {401., -0.0932691}, {402., -0.0256678}, {403.,
-0.144043}, {404., -0.0759929}, {405., -0.180342}, {406., -0.108069},
{407., -0.19612}, {408., -0.1173}, {409., -0.18832}, {410.,
-0.102518}, {411., -0.157539}, {412., -0.0660755}, {413., -0.107918},
{414., -0.0136807}, {415., -0.0464492}, {416., 0.0464747}, {417.,
  0.0179198}, {418., 0.105017}, {419., 0.0756585}, {420.,
  0.152679}, {421., 0.118173}, {422., 0.181733}, {423.,
  0.139016}, {424., 0.18726}, {425., 0.134925}, {426.,
  0.167834}, {427., 0.106425}, {428., 0.125814}, {429.,
  0.0577215}, {430., 0.0671018}, {431., -0.00385498}, {432.,
  0.000219952}, {433., -0.0688669}, {434., -0.0649186}, {435.,
-0.127232}, {436., -0.118452}, {437., -0.169729}, {438., -0.152131},
{439., -0.189374}, {440., -0.160586}, {441., -0.182578}, {442.,
-0.142166}, {443., -0.149741}, {444., -0.099308}, {445., -0.0952401},
{446., -0.0382347}, {447., -0.0268805}, {448., 0.0319437}, {449.,
  0.0452343}, {450., 0.100521}, {451., 0.110219}, {452.,
  0.156806}, {453., 0.15802}, {454., 0.191772}, {455.,
  0.180978}, {456., 0.199447}, {457., 0.175074}, {458.,
  0.177906}, {459., 0.140602}, {460., 0.129677}, {461.,
  0.0822398}, {462., 0.061442}, {463.,
  0.00845556}, {464., -0.0168737}, {465., -0.0697096}, {466.,
-0.0935255}, {467., -0.140259}, {468., -0.156716}, {469., -0.19204},
{470., -0.196396}, {471., -0.216492}, {472., -0.205841}, {473.,
-0.20902}, {474., -0.182793}, {475., -0.169808}, {476., -0.129922},
{477., -0.103938}, {478., -0.0545619}, {479., -0.0207499}, {480.,
  0.0322744}, {481., 0.0674731}, {482., 0.117472}, {483.,
  0.147288}, {484., 0.18777}, {485., 0.206127}, {486.,
  0.23179}, {487., 0.234258}, {488., 0.241839}, {489.,
  0.226368}, {490., 0.215222}, {491., 0.182524}, {492.,
  0.154813}, {493., 0.108343}, {494., 0.0687957}, {495.,
  0.0143147}, {496., -0.0304054}, {497., -0.0856525}, {498.,
-0.127903}, {499., -0.176246}, {500., -0.208549}, {501., -0.243063},
{502., -0.259259}, {503., -0.274866}, {504., -0.271104}, {505.,
-0.265419}, {506., -0.24084}, {507., -0.214644}, {508., -0.171613},
{509., -0.128854}, {510., -0.0727199}, {511., -0.0200219}, {512.,
  0.0415971}, {513., 0.0958504}, {514., 0.15416}, {515.,
  0.201048}, {516., 0.247392}, {517., 0.278816}, {518.,
  0.306018}, {519., 0.31598}, {520., 0.319516}, {521.,
  0.305132}, {522., 0.28394}, {523., 0.246009}, {524.,
  0.202792}, {525., 0.145825}, {526., 0.0867964}, {527.,
  0.0184681}, {528., -0.0474663}, {529., -0.117369}, {530.,
-0.179874}, {531., -0.240882}, {532., -0.289729}, {533., -0.332283},
{534., -0.358939}, {535., -0.375913}, {536., -0.374939}, {537.,
-0.362851}, {538., -0.332878}, {539., -0.292595}, {540., -0.236719},
{541., -0.173516}, {542., -0.0990178}, {543., -0.0219648}, {544.,
  0.0606288}, {545., 0.139903}, {546., 0.218306}, {547.,
  0.287301}, {548., 0.349294}, {549., 0.396535}, {550.,
  0.431867}, {551., 0.448737}, {552., 0.450813}, {553.,
  0.433029}, {554., 0.400112}, {555., 0.348595}, {556.,
  0.28431}, {557., 0.205297}, {558., 0.118328}, {559.,
  0.0226619}, {560., -0.0743484}, {561., -0.172682}, {562.,
-0.264902}, {563., -0.350793}, {564., -0.423405}, {565., -0.482918},
{566., -0.523422}, {567., -0.546026}, {568., -0.546306}, {569.,
-0.52669}, {570., -0.484506}, {571., -0.423674}, {572., -0.343314},
{573., -0.248761}, {574., -0.14071}, {575., -0.0255907}, {576.,
  0.0947767}, {577., 0.213399}, {578., 0.327959}, {579.,
  0.431569}, {580., 0.522131}, {581., 0.593572}, {582.,
  0.644715}, {583., 0.670943}, {584., 0.672574}, {585.,
  0.646916}, {586., 0.596142}, {587., 0.519691}, {588.,
  0.421663}, {589., 0.303534}, {590., 0.171091}, {591.,
  0.0274448}, {592., -0.120464}, {593., -0.268555}, {594., -0.40949},
{595., -0.539061}, {596., -0.650406}, {597., -0.740084}, {598.,
-0.802574}, {599., -0.836016}, {600., -0.836948}, {601., -0.805691},
{602., -0.74129}, {603., -0.646533}, {604., -0.523063}, {605.,
-0.376045}, {606., -0.209422}, {607., -0.0302509}, {608.,
  0.155889}, {609., 0.340862}, {610., 0.518405}, {611.,
  0.680345}, {612., 0.820842}, {613., 0.932805}, {614.,
  1.01191}, {615., 1.05319}, {616., 1.05474}, {617., 1.01447}, {618.,
  0.933518}, {619., 0.813018}, {620., 0.657282}, {621.,
  0.470577}, {622., 0.260027}, {623.,
  0.0324338}, {624., -0.203088}, {625., -0.438233}, {626.,
-0.663088}, {627., -0.869166}, {628., -1.04715}, {629., -1.18981},
{630., -1.28983}, {631., -1.34253}, {632., -1.34378}, {633.,
-1.29247}, {634., -1.1884}, {635., -1.03451}, {636., -0.834757},
{637., -0.596017}, {638., -0.325927}, {639., -0.0345758}, {640.,
  0.267746}, {641., 0.569054}, {642., 0.857925}, {643.,
  1.12216}, {644., 1.35101}, {645., 1.53391}, {646., 1.66257}, {647.,
  1.72975}, {648., 1.73119}, {649., 1.66428}, {650., 1.52977}, {651.,
  1.33034}, {652., 1.07204}, {653., 0.762723}, {654.,
  0.413182}, {655.,
  0.0354996}, {656., -0.356095}, {657., -0.746971}, {658., -1.12143},
{659., -1.46446}, {660., -1.76123}, {661., -1.99876}, {662.,
-2.16542}, {663., -2.25249}, {664., -2.25356}, {665., -2.1659},
{666., -1.98966}, {667., -1.72897}, {668., -1.39095}, {669.,
-0.986456}, {670., -0.528892}, {671., -0.0346511}, {672.,
  0.478289}, {673., 0.990168}, {674., 1.48099}, {675.,
  1.93047}, {676., 2.31965}, {677., 2.63087}, {678., 2.8493}, {679.,
  2.96287}, {680., 2.96355}, {681., 2.84714}, {682., 2.6142}, {683.,
  2.2696}, {684., 1.82309}, {685., 1.28846}, {686., 0.683763}, {687.,
  0.0301875}, {688., -0.648119}, {689., -1.32544}, {690., -1.97488},
{691., -2.56995}, {692., -3.08506}, {693., -3.49709}, {694.,
-3.7859}, {695., -3.93569}, {696., -3.93532}, {697., -3.77939},
{698., -3.46822}, {699., -3.00844}, {700., -2.41261}}
 ]
%\dataplot[linecolor=red,linewidth=0.8pt,plotstyle=line]{\mydata}

\psline[linecolor=gray]{->}(200,0)(750,0)
\psline[linecolor=gray]{->}(250,-4)(250,4)
\multirput(300,-0.15)(100,0){5}{\psline[linecolor=gray](0,0)(0,0.3)}
\multirput(245,-3)(0,1){7}{\psline[linecolor=gray](0,0)(10,0)}
\rput(300,-0.6){$300$}
\rput(400,-0.6){$400$}
\rput(500,-0.6){$500$}
\rput(600,-0.6){$600$}
\rput(700,-0.6){$700$}

\rput(190,3){$\phantom{-}3 \times 10^{-10}$}
\rput(190,2){$\phantom{-}2 \times 10^{-10}$}
\rput(190,1){$\phantom{-}1 \times 10^{-10}$}
\rput(190,-1){$-1 \times 10^{-10}$}
\rput(190,-2){$-2 \times 10^{-10}$}
\rput(190,-3){$-3 \times 10^{-10}$}

\rput(750,-0.7){$n$}

\dataplot[linecolor=darkbrn,linewidth=0.8pt,plotstyle=dots]{\mydata}

%\rput(600,2.9){$1-\frac{W_1(n,n)}{p(n)}$}

\end{pspicture}
\caption{The function $1-W_1(n,n)/p(n)$ for $250 \lqs n \lqs 700$}\label{afig}
\end{center}
\end{figure}
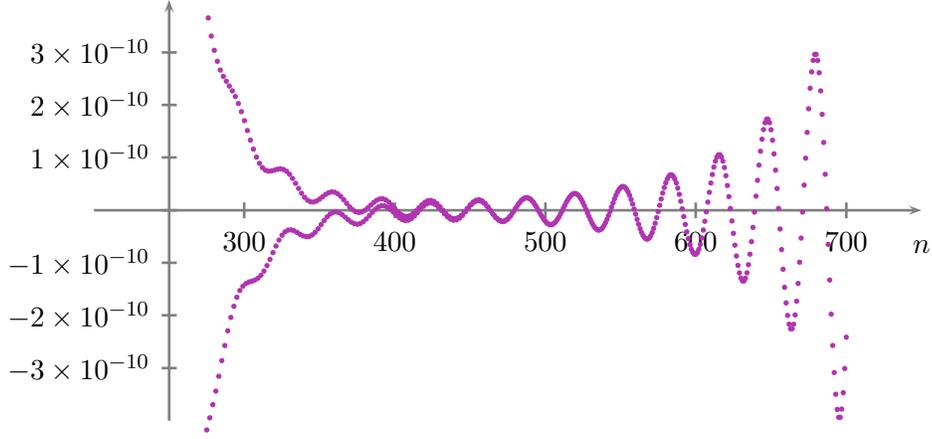

%*************************************

In Figure \ref{afig} we see how closely $W_1(n,n)$ matches $p(n)$ for relatively small $n$. We know that the first waves oscillate with periodicity of length $2\pi/V \approx 31.96311$ for large $n$ and this periodicity is already  visible in the figure.
Our computations show, for example, that
\begin{equation*}
    \left| 1-\frac{W_1(n,n)}{p(n)}\right| < 2\times 10^{-6} \quad \text{ for } \quad 100\lqs n \lqs 700.
\end{equation*}
The smallest value of $|1-W_1(n,n)/p(n)|$ for $1\lqs n \lqs 700$ occurs when $n=432$, giving the very accurate approximation
\begin{align*}
    W_1(432,432) & \approx \underline{46\,647\,863\,284\,22} 8\,241\,960.9,\\
    p(432)& = \underline{46\,647\,863\,284\,22} 9\,267\,991
\end{align*}

For $\lambda=2$ the initial agreement between $p_N(2N)$ and $W_1(N,2N)$ is even more impressive. For example,
\begin{equation*}
    \left| 1-\frac{W_1(N,2N)}{p_N(2N)}\right| \approx 5.86\times 10^{-21} \quad \text{for} \quad N=1000.
\end{equation*}
For $N=2000$ it is  $\approx 2.54\times 10^{-12}$.
We  computed the restricted partition $p_N(2N)$ above by relating it to $p(n)$ with the formula
\begin{equation*}
    p_N(2N) = p(2N)-\sum_{m=0}^{N-1}p(m).
\end{equation*}
This follows since we must remove from the partitions in $p(2N)$ those with largest part of size $2N-m$.  The  partition function $p(n)$ may be easily calculated using \eqref{hrr}.

Examples of Theorem \ref{ma2} with $\lambda=2$ are shown in Table \ref{tbl22}. This shows that after about $N=2900$ the first wave begins to grow more rapidly than $p_N(2N)$ and follows the expected asymptotics. See also the bottom row in Table \ref{tgz} for more on the $N=3300$ case.
\begin{table}[ht]
\begin{center}
\begin{tabular}{c | c | c | c}
$N$ & $p_N(2N)$ & $W_1(N,2N)$ &  $\Re\left[ (2  z_0 e^{-5\pi i z_0})\frac{w_0^{-N}}{N^2} \right]$ \\ \hline
$2700$ & $\mathbf{1.94\times 10^{77}}$ & $ \mathbf{\phantom{-}1.92\times 10^{77}}$  & $-1.16\times 10^{75}$  \\
$2800$ & $\mathbf{5.93\times 10^{78}}$ & $ \mathbf{\phantom{-}5.49\times 10^{78}}$  & $-4.33\times 10^{77}$  \\
$2900$ & $1.71\times 10^{80}$ & $ \phantom{-}4.63\times 10^{80}$  & $\phantom{-}3.20\times 10^{80}$  \\
$3000$ & $4.67\times 10^{81}$ & $ \mathbf{\phantom{-}6.60\times 10^{83}}$  & $\mathbf{\phantom{-}6.83\times 10^{83}}$  \\
$3100$ & $1.21\times 10^{83}$ & $ \mathbf{\phantom{-}5.59\times 10^{86}}$ &   $\mathbf{\phantom{-}5.70\times 10^{86}}$  \\
$3200$ & $2.96\times 10^{84}$ & $ \mathbf{\phantom{-}1.83\times 10^{89}}$  & $\mathbf{\phantom{-}1.78\times 10^{89}}$  \\
$3300$ & $6.92\times 10^{85}$ & $ \mathbf{-1.89\times 10^{92}}$  & $\mathbf{-2.03\times 10^{92}}$
\end{tabular}
\caption{Comparing  $p_N(2N)$, $W_1(N,2N)$ and the asymptotics from Theorem \ref{ma2}.} \label{tbl22}
\end{center}
\end{table}

We can offer the following explanation for why there is such good initial agreement between $ p_N(\lambda N)$ and the first waves. From \eqref{ptg} in the proof of Theorem \ref{ma2} we know that
\begin{equation} \label{usio}
    \sum_{k=1}^{100} W_k(N,\lambda N)=p_N(\lambda N) +\frac{e^{UN}}{N^2}\psi_\lambda \cdot \sin\Bigl(\tau_\lambda +VN\Bigr)+ O\left(\frac{e^{UN}}{N^{3}}\right),
\end{equation}
recalling \eqref{abuv}, \eqref{presxx} and \eqref{frbr}. For small $N$, $ p_N(\lambda N)$ is the largest term on the right of \eqref{usio}. This means the first waves will be close to $ p_N(\lambda N)$. However $ p_N(\lambda N) \ll  \exp(\bigl(2\pi \sqrt{|\lambda|/6}\bigr)\sqrt{N})$ as in \eqref{ptg2} which implies that, for any given $\lambda$, the second term on the right of \eqref{usio} will always eventually dominate and hence provide the asymptotics for the first waves as $N \to \infty$.

\section{Future work} \label{fur}
As we noted in Section \ref{maru}, the main results of Theorems \ref{ma1}, \ref{ma2}, \ref{ma3} and Corollary \ref{cma2} should be true with the sum over the first $100$ waves replaced by just the first wave. Hence Theorem \ref{ma2} becomes:

\begin{conj} \label{conja}
Let $\lambda^+$ be a positive real number. Suppose $N \in \Z_{\gqs 1}$ and  $\lambda N \in \Z$ for $\lambda$ satisfying $|\lambda| \lqs \lambda^+$.
Then
\begin{equation*}
    W_1(N,\lambda N) = \Re\left[\frac{w_0^{-N}}{N^{2}} \left( a_{0}(\lambda)+\frac{a_{1}(\lambda)}{N}+ \dots +\frac{a_{m-1}(\lambda)}{N^{m-1}}\right)\right] + O\left(\frac{|w_0|^{-N}}{N^{m+2}}\right)
\end{equation*}
where the coefficients $a_{t}(\lambda)$ are given in \eqref{btys} and the implied constant depends only on  $\lambda^+$ and $m$.
\end{conj}

We have already seen evidence for this conjecture in Table \ref{tgz}.
Its proof  would require an improvement in Proposition \ref{by}; the  techniques of \cite{Sze51} or \cite{DrGe} might allow a careful enough estimate of $W_2+W_3+ \cdots +W_{100}$. Strengthening Proposition \ref{by} would also  increase  the allowable range of $\lambda$ for Theorem \ref{ma3}.

Numerical experiments reveal that $W_2(N,\lambda N)$ matches $(-1)^{N+1}\mathcal D_1(N,-\lambda N)$ for large $N$ and we expect their asymptotic expansions are the same:
\begin{conj} \label{conjd}
Let $\lambda^+$ be a positive real number. Suppose $N \in \Z_{\gqs 1}$ and  $\lambda N \in \Z$ for $\lambda$ satisfying $|\lambda| \lqs \lambda^+$.
Then with $\overline{N}$ denoting $N \bmod 2$, we have
\begin{equation*}
    W_2(N,\lambda N) = (-1)^{N+1}\Re\left[\frac{w_0^{-N/2}}{N^{2}} \left( d_{0}\bigl(\lambda, \overline{N}\bigr) +\frac{d_{1}\bigl(\lambda, \overline{N}\bigr)}{N}+ \dots +\frac{d_{m-1}\bigl(\lambda, \overline{N}\bigr)}{N^{m-1}}\right)\right] + O\left(\frac{|w_0|^{-N/2}}{N^{m+2}}\right)
\end{equation*}
for an implied constant depending only on  $\lambda^+$ and $m$. The functions $d_{t}\bigl(\lambda, \overline{N}\bigr)$ are given in \eqref{dso} and \eqref{dso2}.
\end{conj}
%In Table \ref{w2d} we compare some values of $W_2(N,\lambda N)$ to $-D_1(N,-

\begin{table}[ht]
{\footnotesize
\begin{center}
\begin{tabular}{cc|cccc|c}
$N$ & $\lambda$ & $m=1$ & $m=2$ & $m=3$ &  $m=5$ & $W_2(N,\lambda N)$  \\ \hline
$3000$ & $1/3$ & $\phantom{-}6.13580 \times 10^{37}$ &  $\phantom{-}6.18769 \times 10^{37}$ & $\phantom{-}6.18681 \times 10^{37}$ &  $\phantom{-}6.18680 \times 10^{37}$ &  $\phantom{-}6.18680 \times 10^{37}$\\
$3000$ & $1$ & $\phantom{-}2.20860 \times 10^{36}$ &  $-2.19459 \times 10^{35}$ & $-1.79624 \times 10^{35}$ &  $-1.79070 \times 10^{35}$ &  $-1.79070 \times 10^{35}$ \\
$3000$ & $5/3$ & $-1.84871 \times 10^{38}$ &  $-1.77143 \times 10^{38}$ & $-1.77234 \times 10^{38}$ &  $-1.77239 \times 10^{38}$ &  $-1.77190 \times 10^{38}$
\end{tabular}
\caption{The  approximations of Conjecture \ref{conjd} to $W_2(N,\lambda N)$.} \label{w2tb}
\end{center}}
\end{table}
Table \ref{w2tb} gives examples and the agreement for $N$ odd is similar. Conjecture \ref{conjd} is the analog of \cite[Conj. 6.4]{OS1} for the Rademacher coefficients.
Conjectures \ref{conja} and \ref{conjd} together imply that $W_2(N,\lambda N)$ squared is approximately  $W_1(N,\lambda N)/N^2$ for large $N$.

Comparing  Tables \ref{c2n1} and \ref{e1n1}, we see that the entries for $\mathcal E_1(N,-\lambda N)$ in Table \ref{e1n1} are exactly $3$ times the entries for $\mathcal C_2(N,-\lambda N)$ in Table \ref{c2n1}. The following is based on further numerical evidence.

\begin{conj} %\label{conjd}
The quantities $\mathcal E_1(N,-\lambda N)$ and $3\mathcal C_2(N,-\lambda N)$, given by
\begin{equation*}
      2 \Re \sum_{ \frac{N}{3}  <k \lqs \frac{N}{2}}  Q_{1k(-\lambda N)}(N)
    \qquad \text{and} \qquad
      6\Re \sum_{ \frac{2N}{3}  <k \lqs N, \ k\text{ odd}}  Q_{2k(-\lambda N)}(N)
\end{equation*}
respectively, have the same asymptotic expansion as $N \to \infty$. In other words, recalling \eqref{cso}, \eqref{eso}, we have $e_t(\lambda)=3c_t(\lambda)$ for all $t \in \Z_{\gqs 0}$.
\end{conj}

We may finally ask how  our asymptotic results extend to other examples of the more general waves of Section \ref{gend}.

{\small
\bibliography{wavedata}
}

{\small %\footnotesize
\vskip 5mm
\noindent
\textsc{Dept. of Math, The CUNY Graduate Center, 365 Fifth Avenue, New York, NY 10016-4309, U.S.A.}

\noindent
{\em E-mail address:} \texttt{cosullivan@gc.cuny.edu}
}

\end{document}